\documentclass[11pt, a4paper]{article}
\usepackage[utf8]{inputenc}
\usepackage[usenames,dvipsnames]{xcolor}
\usepackage{amsmath}
\usepackage{amsthm}
\usepackage{amsfonts}
\usepackage{amssymb}
\usepackage{array}  
\usepackage{graphicx} 
\usepackage{color}
\usepackage[english]{babel}
\usepackage{mathrsfs}
\usepackage{graphicx}
\usepackage{dsfont}
\definecolor{orangebis}{rgb}{0.99,0.25,0.00}
\definecolor{greenbis}{rgb}{0.10,0.85,0.10}
\definecolor{bluebis}{rgb}{0.10,0.30,0.99}
\usepackage[final]{hyperref}   
\usepackage{subcaption}

\usepackage{multicol}

\usepackage{lipsum}
\usepackage{geometry}
\author{Alejandro Rivera and Hugo Vanneuville}
\title{}
\date{}

\theoremstyle{plain}
\newtheorem{thm}{Theorem}[section]
\newtheorem{prop}[thm]{Proposition}
\newtheorem{lem}[thm]{Lemma}

\newtheorem{cor}[thm]{Corollary}
\newtheorem{claim}[thm]{Claim}

\newtheorem{theorem}[thm]{Theorem}

\newtheorem{proposition}[thm]{Proposition}
\newtheorem{conjecture}[thm]{Conjecture}
\newtheorem{lemma}[thm]{Lemma}

\theoremstyle{definition}
\newtheorem{defi}[thm]{Definition}

\newtheorem{notation}[thm]{Notation}
\newtheorem{condition}[thm]{Condition}
\newtheorem{remark}[thm]{Remark}

\newcommand{\lobyh}{4.5}
\newcommand{\tassion}{1.1}
\newcommand{\harris}{4.6}
\newcommand{\duality}{\text{\textup{A.11}}}
\newcommand{\harrisdiscrete}{\text{\textup{B.7}}}
\newcommand{\tassiondiscrete}{\text{\textup{B.2}}}
\newcommand{\polynomially}{\text{\textup{B.6}}}
\newcommand{\transversality}{\text{\textup{A.9}}}
\newcommand{\quasiindependence}{1.12}

\newcommand{\N}{\mathbb{N}}
\newcommand{\R}{\mathbb{R}}

\newcommand{\Z}{\mathbb{Z}}

\newcommand{\Pro}{\mathbb{P}}
\newcommand{\prob}{\mathbb{P}}
\newcommand{\E}{\mathbb{E}}
\newcommand{\T}{\mathbb{T}}

\newcommand{\cross}{\text{\textup{Cross}}}
\newcommand{\cov}{\text{\textup{Cov}}}
\newcommand{\var}{\text{\textup{Var}}}
\newcommand{\piv}{\text{\textup{Piv}}}
\newcommand{\arm}{\text{\textup{Arm}}}
\newcommand{\Circ}{\text{\textup{Circ}}}
\newcommand{\Infl}{\text{\textup{Infl}}}
\newcommand{\Piv}{\text{\textup{Piv}}}
\newcommand{\ann}{\text{\textup{Ann}}}

\newcommand{\fold}{\text{\textup{Fold}}}

\newcommand{\fivecross}{\text{\textup{MultiCross}}}
\newcommand{\vp}{\overrightarrow{p}}

\renewcommand{\Z}{\mathbb{Z}}
\renewcommand{\N}{\mathbb{N}}

\def\T{\mathbb{T}}

\newcommand{\un}{\mathds{1}}

\newcommand{\cond}{\, \Big| \,}
\renewcommand{\textbf}[1]{\begingroup\bfseries\mathversion{bold}#1\endgroup}
\def\calD{\mathcal{D}}
\def\calE{\mathcal{E}}

\def\calG{\mathcal{G}}

\def\calN{\mathcal{N}}

\def\calT{\mathcal{T}}

\def\calV{\mathcal{V}}

\def\sfF{\mathsf{F}}

\def\var{\mathop{\mathrm{Var}}}

\def\cov{\mathrm{Cov}}

\def\E{\mathbb{E}} 

\def \eps {\varepsilon}

\def\<#1{\langle #1\rangle}

\def\bi{\begin{itemize}}  
\def\ei{\end{itemize}}
\def\bnum{\begin{enumerate}} 
\def\enum{\end{enumerate}}
\def\ni{\noindent}
\def\bf{\bfseries}

\numberwithin{equation}{section}
\title{The critical threshold for Bargmann-Fock percolation}
\author{Alejandro Rivera\thanks{Univ. Grenoble Alpes, UMP5582, Institut Fourier, 38000 Grenoble, France, supported by the ERC grant Liko No 676999} \hspace{1cm} Hugo Vanneuville\thanks{Univ. Lyon 1, Institut Camille Jordan, 69100 Villeurbanne, France, supported by the ERC grant Liko No 676999}}
\date{}

\begin{document}

\maketitle
\begin{abstract}
In this article, we study the excursion sets $\calD_p=f^{-1}([-p,+\infty[)$ where $f$ is a natural real-analytic planar Gaussian field called the Bargmann-Fock field. More precisely, $f$ is the centered Gaussian field on $\R^2$ with covariance $(x,y) \mapsto \exp(-\frac{1}{2}|x-y|^2)$. Alexander has proved that, if $p \leq 0$, then a.s. $\calD_p$ has no unbounded component. We show that conversely, if $p>0$, then a.s. $\calD_p$ has a unique unbounded component. As a result, the critical level of this percolation model is $0$. We also prove exponential decay of crossing probabilities under the critical level. To show these results, we rely on a recent box-crossing estimate by Beffara and Gayet. We also develop several tools including a KKL-type result for biased Gaussian vectors (based on the analogous result for product Gaussian vectors by Keller, Mossel and Sen) and a sprinkling inspired discretization procedure. These intermediate results hold for more general Gaussian fields, for which we prove a discrete version of our main result.
\end{abstract}

\bigskip

\begin{figure}[!h]
\centering
\begin{subfigure}{.5\textwidth}
  \centering
  \includegraphics[width=0.9\textwidth]{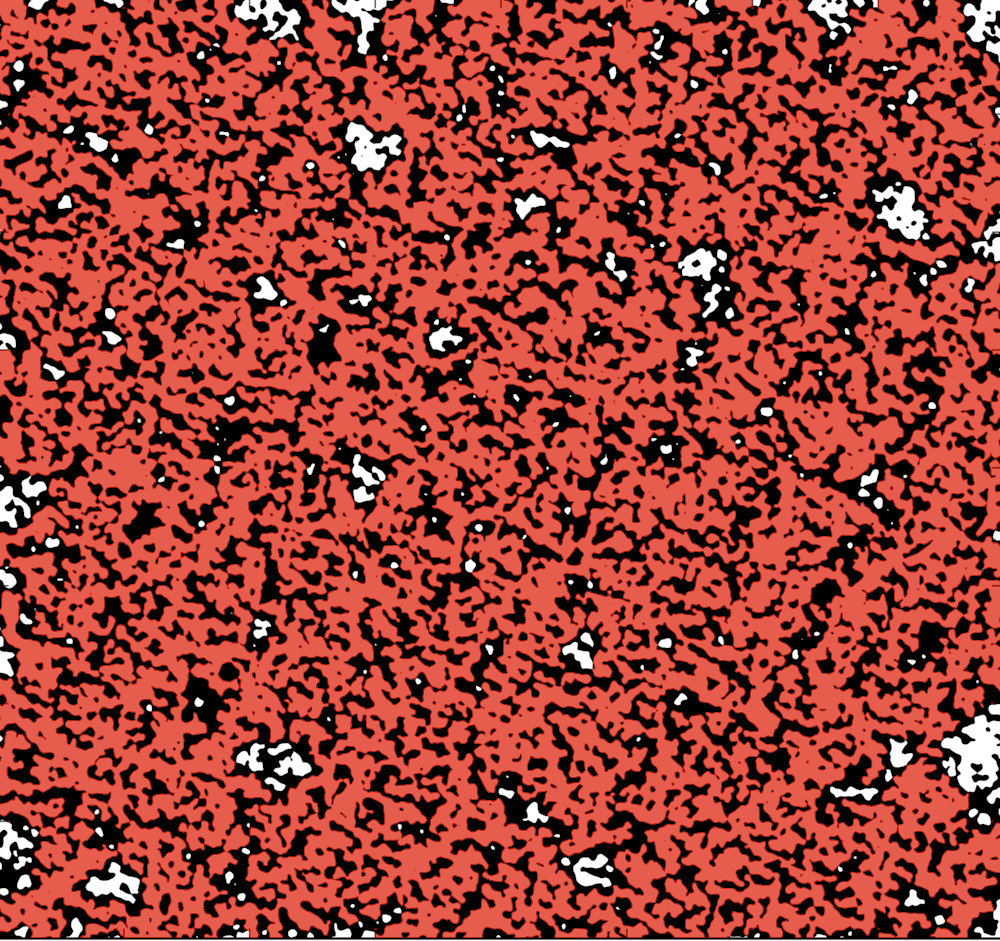}
  \label{fig:sub1}
\end{subfigure}%
\begin{subfigure}{.5\textwidth}
  \centering
  \includegraphics[width=0.9\textwidth]{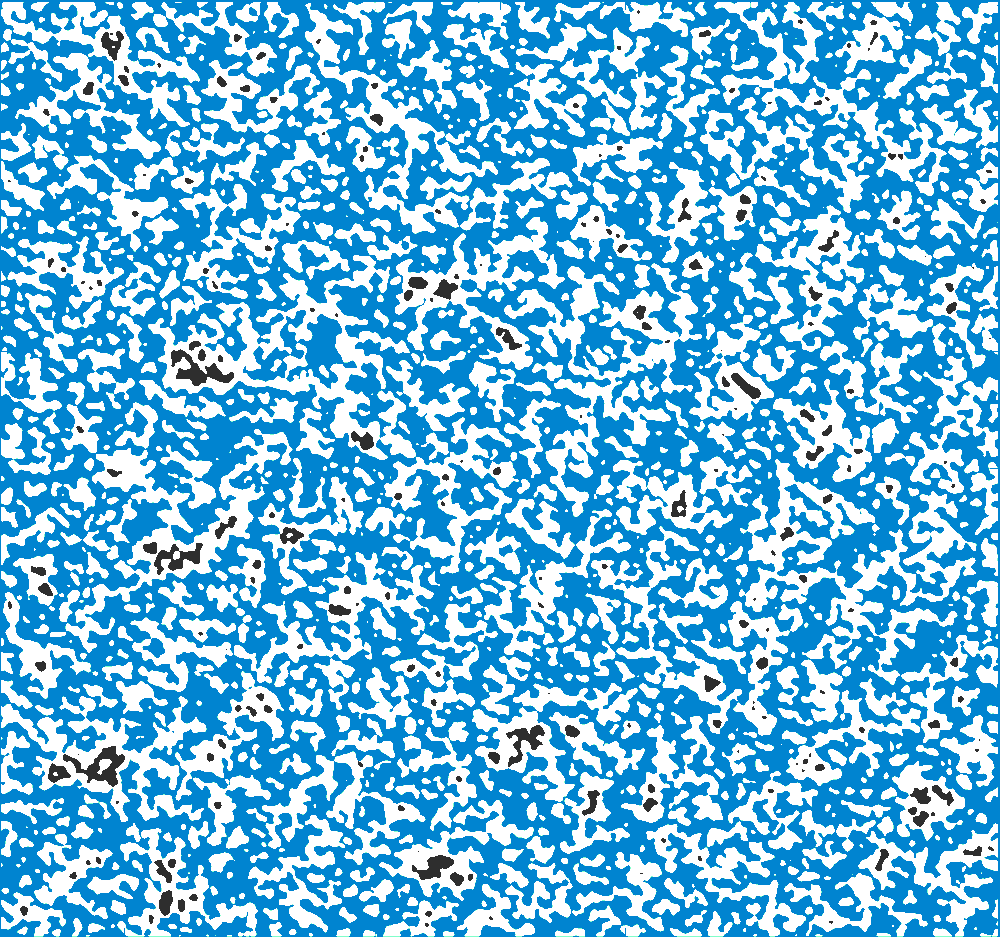}
  \label{fig:sub2}
\end{subfigure}
\caption{Percolation for an approximation of the Bargmann-Fock field $f$. On the left hand side, the region $f \geq 0.1$ is colored in black, the small components of the region $f <  0.1$ are colored in white, and the giant component of the region $f <  0.1$ is colored in red. On the right hand side, the small components of the region $f \geq -0.1$ are colored in black, the giant component of the region $f \geq -0.1$ is colored in blue, and the region $f < -0.1$ is colored in white. The two pictures correspond to the same sample of $f$.}
\label{fig:test}
\end{figure}

\tableofcontents

\section{Main results}\label{s.main_results}

In this article, we study the geometry of excursion sets of a planar centered Gaussian field $f$. The \textbf{covariance function} of $f$ is the function $K:\R^2\times\R^2\rightarrow\R$ defined by:
\[
\forall x,y\in\R^2,\ K(x,y)=\E[f(x)f(y)] \, .
\]
We assume that $f$ is normalized so that for each $x\in\R^2$, $K(x,x)=\var(f(x))=1$, that it is non-degenerate (i.e. for any pairwise distinct $x_1,\dots,x_k\in\R^2$, $(f(x_1),\cdots,f(x_k))$ is non-degenerate), and that it is a.s. continuous and stationary. In particular, there exists a strictly positive definite continuous function $\kappa:\R^2\rightarrow[-1,1]$ such that $\kappa(0)=1$ and, for each $x,y\in\R^2$, $K(x,y)=\kappa(x-y)$. For our main results (though not for all the intermediate results), we will also assume that $f$ is positively correlated, which means that $\kappa$ takes only non-negative values. We will also refer to $\kappa$ as covariance function when there is no possible ambiguity. For each $p\in\R$ we call \textbf{level set} of $f$ the random set $\mathcal{N}_p:=f^{-1}(-p)$ and \textbf{excursion set} of $f$ the random set $\mathcal{D}_p:=f^{-1}([-p,+\infty[)$.\footnote{This convention, while it may seem counterintuitive, is convenient because it makes $\calD_p$ increasing both in $f$ and in $p$. See Section~\ref{s.strategy} for more details.}\\

These sets have been studied through their connections to percolation theory (see~\cite{molchanov1983percolationi,molchanov1983percolationii,molchanov1986percolationiii}, \cite{alex_96}, \cite{bs_07}, \cite{bg_16}, \cite{bm_17}, \cite{bmw_17}, \cite{rv_rsw}). In this theory, one wishes to determine whether of not there exist unbounded connected components of certain random sets. So far, we know that $\calD_0$ has a.s. only bounded components for a very large family of positively correlated Gaussian fields:
\begin{thm}[Theorem~2.2 of \cite{alex_96}]\label{t.alex}
Assume that for each $x\in\R^2$, $\kappa(x)\geq 0$, that $f$ is a.s. $C^1$ and ergodic with respect to translations. Assume also that for each $p\in\R$, $f$ has a.s. no critical points at level $p$. Then, a.s. all the connected components of $\calD_0$ are bounded.
\end{thm}
\begin{proof}
By \cite{pitt1982positively}, the fact that $\kappa$ is non-negative implies that $f$ satisfies the FKG inequality\footnote{What Alexander calls the FKG inequality is the positive association property of any finite dimensional marginal distributions of the field, see the beginning of his Section 2.} so we can apply Theorem 2.2 of \cite{alex_96}. Hence, all its level lines are bounded. By ergodicity, $\calD_0$ has either a.s. only bounded connected components or a.s. at least one unbounded connected component. Since $f$ has a.s. no critical points at level $0$, a.s., the boundary of $\calD_0$ equals $\calN_0$ and is a $C^1$ submanifold of $\R^2$. If $\calD_0$ had a.s. an unbounded connected component, then, by symmetry (since $f$ is centered) this would also be the case for $\overline{\R^2\setminus\calD_0}$. But this would imply that $\calN_0$ has an unbounded connected component, thus contradicting Theorem 2.2 of \cite{alex_96}.-
\end{proof}
More recently, Beffara and Gayet~\cite{bg_16} have proved a more quantitative version of Theorem~\ref{t.alex} which holds for a large family of positively correlated stationary Gaussian fields such that $\kappa(x) = O(|x|^{-\alpha})$ for some $\alpha$ sufficiently large. In~\cite{rv_rsw}, the authors of the present paper have revisited the results by~\cite{bg_16} and weaken the assumptions on $\alpha$. More precisely, we have the following:

\begin{thm}[\cite{bg_16} for $\alpha$ sufficiently large,\cite{rv_rsw}]\footnote{More precisely, this is Propositions~$\lobyh$ and $\harris$ of~\cite{rv_rsw}. Moreover, this is Theorem~5.7 of~\cite{bg_16} for $\alpha$ sufficiently large and with slightly different assumptions on the differentiability and on the non-degeneracy of $\kappa$.}\label{t.c_RSW}
Assume that $f$ is a non-degenerate, centered, normalized, continuous, stationary, positively correlated planar Gaussian field that satisfies the symmetry assumption Condition~\ref{a.std} below. Assume also that $\kappa$ satisfies the differentiability assumption Condition~\ref{a.decay} below and that $\kappa(x) \leq C |x|^\alpha$ for some $C<+\infty$ and $\alpha > 4$. Then, there exist $C'=C'(\kappa)<+\infty$ and $\delta=\delta(\kappa)>0$ such that for each $r>0$, the probability that there exists a connected component of $\calD_0$ which connects $0$ to a point at distance $r$ is at most $C' r^{-\delta}$. In particular, a.s. all the connected components of $\calD_0$ are bounded.
\end{thm}

A remaining natural question is whether or not, for $p>0$, the excursion set $\mathcal{D}_p$ has an unbounded component. Our main result (Theorem~\ref{t.main} below) provides an answer to this question for a specific, natural choice of $f$, arising naturally from real algebraic geometry: the Bargmann-Fock model, that we now introduce. \textbf{The planar Bargmann-Fock field} is defined as follows. Let $(a_{i,j})_{i,j\in\N}$ be a family of independent centered Gaussian random variables of variance $1$. For each $x=(x_1,x_2)\in\R^2$, the Bargmann-Fock field at $x$ is:
\[
f(x)=e^{-\frac{1}{2}|x|^2}\sum_{i,j\in\N}a_{i,j}\frac{x_1^ix_2^j}{\sqrt{i!j!}} \, .
\]
The sum converges a.s. in $C^\infty(\R^2)$ to a random analytic function. Moreover, for each $x,y\in\R^2$:
\[
K(x,y)=\E[f(x)f(y)]=\exp\Big(-\frac{1}{2}|x-y|^2\Big) \, .
\]
For a discussion of the relation of the Bargmann-Fock field with algebraic geometry, we refer the reader to the introduction of~\cite{bg_16}. Theorem~\ref{t.c_RSW} applies to the Bargmann-Fock model (see Subsection~\ref{ss.cond} for the non-degeneracy condition). Hence, if $p \leq 0$ then a.s. all the connected components of $\calD_p$ are bounded. In this paper, we prove that on the contrary if $p>0$ then a.s. $\calD_p$ has a unique unbounded component, thus obtaining the following result:
\begin{theorem}\label{t.main}
Let $f$ be the planar Bargmann-Fock field. Then, the probability that $\mathcal{D}_p$ has an unbounded connected component is $1$ if $p>0$, and $0$ otherwise. Moreover, if it exists, such a component is a.s. unique.
\end{theorem}
As a result, the ``critical threshold'' of this continuum percolation model is $p=0$. Before saying a few words about the proof of Theorem~\ref{t.main}, let us state the result which is at the heart of the proof of Theorem~\ref{t.c_RSW}, both in~\cite{bg_16} and in~\cite{rv_rsw}. This result is a \textbf{box-crossing estimate}, which is an analog of the classical Russo-Seymour-Welsh theorem for planar percolation. This was proved in~\cite{bg_16} for a large family of positively correlated stationary Gaussian fields such that $\kappa(x) = O(|x|^{-325})$. In~\cite{bm_17}, Beliaev and Muirhead have lowered the exponent $\alpha=325$ to any $\alpha > 16$. In~\cite{rv_rsw}, we have obtained that such a result holds with any exponent $\alpha > 4$. More precisely, we have the following:

\begin{thm}[\cite{bg_16} for $\alpha > 325$,~\cite{bm_17} for $\alpha > 16$,\cite{rv_rsw}]\footnote{More precisely, this is Theorem~$\tassion$ of~\cite{rv_rsw}. Moreover, this is Theorem~4.9 of~\cite{bg_16} (resp. Theorem~1.7 of~\cite{bm_17}) for $\alpha \geq 325$ (resp. $\alpha > 16$) and with slightly different assumptions on the differentiability and on the non-degeneracy of $\kappa$.}\label{t.RSW}
With the same hypotheses as Theorem~\ref{t.c_RSW}, for every $\rho  \in ]0,+\infty[$ there exists $c=c(\rho,\kappa)>0$ such that for each $R \in ]0,+\infty[$, the probability that there is a continuous path in $\mathcal{D}_0\cap[0,\rho R]\times [0,R]$ joining the left side of $[0,\rho R]\times [0,R]$ to its right side is at least $c$. Moreover, there exists $R_0 = R_0(\kappa) < +\infty$ such that the same result holds for $\mathcal{N}_0$ as long as $R\geq R_0$.
\end{thm}

In order to prove our main result Theorem~\ref{t.main}, we will use a discrete analog of this box-crossing estimate which goes back to~\cite{bg_16} (see Theorem~\ref{t.epsRSW} below). In Section~\ref{s.strategy}, we expose the general strategy of the proof of Theorem~\ref{t.main}. This proof can be summed up as follows: i) we discretize our model as was done in~\cite{bg_16}, ii) we prove that there is a sharp threshold phenomenon at $p=0$ in the discrete model, iii) we return to the continuum. The results at the heart of our proof of a sharp threshold phenomenon for the discrete model are on the one hand Theorem~\ref{t.epsRSW} (the discrete version of Theorem~\ref{t.RSW}) and on the other hand a Kahn-Kalai-Linial (KKL)-type estimate for biased Gaussian vectors (see Theorem~\ref{t.KMSnonproduct}) that we show by using the analogous estimate for product Gaussian vectors proved by Keller, Mossel and Sen in~\cite{keller2012geometric} (the idea to use a KKL theorem to compute the critical point of a percolation model goes back to Bollob\'{a}s and Riordan~\cite{bollobas2006critical,bollobas2006short}, see Subsection~\ref{ss.kesten} for more details). To go back to the continuum, we apply a sprinkling argument to a discretization procedure tailor-made for our setting (see Proposition \ref{p.sprinkling}). This step is especially delicate since Theorem~\ref{t.KMSnonproduct} gives no relevant information when the discretization mesh is too fine (see Subsection~\ref{ss.discretization} for more details).\\
Most of the intermediate results that we will prove work in a much wider setting, see in particular Proposition~\ref{p.mainmoregnrl} where we explain how, for a large family of Gaussian fields $f$, the proof of an estimate on the correlation function would imply that Theorem~\ref{t.main} also holds for $f$. See also Theorem~\ref{t.epsRSW} which is a discrete analog of Theorem~\ref{t.main} for more general Gaussian fields.\\
As in~\cite{bg_16}, we are inspired by tools from percolation theory. Before going any further, let us make a short detour to present the results of planar percolation we used to guide our research. It will be helpful to have this analogy in mind to appreciate our results.\\

\textbf{Planar Bernoulli percolation} is a statistical mechanics model defined on a planar lattice, say $\Z^2$, depending on a parameter $p\in[0,1]$. Consider a family of independent Bernoulli random variables $(\omega_e)_e$ of parameter $p$ indexed by the edges of the graph $\Z^2$. We say that an edge $e$ is black if the corresponding random variable $\omega_e$ equals $1$ and white otherwise. The analogy with our model becomes apparent when one introduces the following classical coupling of the $(\omega_e)_e$ for various values of $p$. Consider a family  $(U_e)_e$ of independent uniform random variables in $[0,1]$ indexed by the set of edges of $\Z^2$. For each $p\in[0,1]$, let $\omega^p_e=\un_{\lbrace U_e\geq 1-p\rbrace}$. Then, the family $((\omega_e^p)_e)_p$ forms a coupling of Bernoulli percolation with parameters in $[0,1]$. In this coupling, black edges are seen as excursion sets of the random field $(U_e)_e$. Theorem~\ref{t.RSW} is the analog of the Russo-Seymour-Welsh (RSW) estimates first proved for planar Bernoulli percolation in~\cite{russo1978note,seymour1978percolation}. We now state the main result of percolation theory on $\Z^2$, a celebrated theorem due to Kesten.

\begin{theorem}[Kesten's Theorem, \cite{kesten1980critical}]\label{t.kesten}
  Consider planar Bernoulli percolation of parameter $p\in[0,1]$ on $\Z^2$. If $p>1/2$, then a.s. there exists an unbounded connected component made of black edges. On the other hand, if $p\leq1/2$ then a.s. there is no unbounded connected component made of black edges.
\end{theorem}

The parameter $p_c=1/2$ is said to be critical for planar Bernoulli percolation on $\Z^2$. It is also known that, if such an unbounded connected component exists, it is a.s. unique. In Theorem~\ref{t.kesten}, the case where $p \leq 1/2$ goes back to Harris~\cite{harris1960lower}. See Subsection~\ref{ss.kesten} where we explain which are the main ingredients of a proof of Kesten's theorem that will inspire us. Kesten's theorem is closely linked with another, more quantitative result:

\begin{theorem}[Exponential decay, \cite{kesten1980critical}]\label{t.ber_exp}
Consider planar Bernoulli percolation with parameter $p>1/2$ on $\Z^2$. Then, for each $\rho>0$, there exists a constant $c=c(p,\rho)>0$ such that for each $R>0$, the probability that there is a continuous path made of black edges in $[0,\rho R]\times[0,R]$ joining the left side of $[0,\rho R]\times[0,R]$ to its right side is at least $1-e^{-cR}$. 
\end{theorem}
The value $p=1/2$ is significant because with this choice of parameter, the induced percolation model on the dual graph of $\Z^2$ has the same law as the initial one. For this reason, $1/2$ is called the self-dual point. In the case of our planar Gaussian model, self-duality arises at the parameter $p=0$ (see the duality properties used in~\cite{bg_16,rv_rsw}, for instance Remark~$\duality$ of~\cite{rv_rsw}). The results on Bernoulli percolation lead us to the following conjecture:

\begin{conjecture}
For centered, normalized, non-degenerate, sufficiently smooth, stationary, isotropic, and positively correlated random fields on $\R^2$ with sufficient correlation decay, the probability that $\mathcal{D}_p$ has an unbounded connected component is $1$ if $p>0$, and $0$ otherwise.
\end{conjecture}

Of course, our Theorem~\ref{t.main} is an answer to the above conjecture for a particular model.  We also have the following analog analog for Theorem \ref{t.ber_exp} that we prove in Subsection~\ref{ss.exp}.

\begin{theorem}\label{t.exp}
Consider the Bargmann-Fock field and let $p>0$. Then, for each $\rho>0$ there exists a constant $c=c(p,\rho)>0$ such that for each $R>0$, the probability that there is a continuous path in $\mathcal{D}_p \cap [0,\rho R]\times[0,R]$ joining the left side of $[0,\rho R]\times[0,R]$ to its right side is at least $1-e^{-cR}$.
\end{theorem}

Contrary to the other results of this paper, the fact that the correlation function of the Bargmann-Fock field decreases more than exponentially fast is crucial to prove Theorem~\ref{t.exp}.

\paragraph{Organization of the paper.} The paper is organized as follows:
\bi 
\item In Section~\ref{s.strategy}, we explain the general strategy of the proof of our phase transition result Theorem~\ref{t.main}. In particular, we explain the discretization we use and we state the main intermediate results including the KKL-type inequality for biased Gaussian vectors mentioned above.
\item In Section~\ref{s.mainpf}, we combine all these intermediate results to prove Theorems~\ref{t.main} and~\ref{t.epsRSW}.
\item In Sections~\ref{s.perco} to~\ref{s.sprinkling}, we prove these intermediate results.
\ei

\paragraph{Related works.} As explained above, the present article is in the continuity of~\cite{bg_16} where the authors somewhat initiate the study of a rigorous connection between percolation theory and the behaviour of nodal lines. In~\cite{bm_17}, the authors optimize the results from~\cite{bg_16} and the authors of the present paper optimize them further in~\cite{rv_rsw}. See also~\cite{bmw_17} where the authors prove a box-crossing estimate for Gaussian fields on the sphere or the torus by adapting the strategy of~\cite{bg_16}. In~\cite{bg_16,bm_17,bmw_17,rv_rsw} (while the approaches differ in some key places), the initial idea is the same, namely to use Tassion's general method to prove box-crossing estimates, which goes back to~\cite{tassion2014crossing}. To apply such a method, we need to have in particular a positive association property and a quasi-independence property. In~\cite{rv_rsw}, we have proved such a quasi-independence property for planar Gaussian fields that we will also use in the present paper, see Claim~\ref{cl.quasi_ind}. We will also rely to other results from~\cite{rv_rsw}, in particular a discrete RSW estimate. As we will explain in Subsection~\ref{ss.discretization}, we could have rather referred to the slightly weaker analogous results from~\cite{bg_16}, which would have been sufficient in order to prove our main result Theorem~\ref{t.main}.\\

The use of a KKL theorem to prove our main result Theorem~\ref{t.main} shows that our work falls within the approach of recent proofs of phase transition that mix tools from percolation theory and tools from the theory of Boolean functions. See~\cite{garban2014noise} for a book about how these theories can combine. Below, we list some related works in this spirit.

\bi 
\item During the elaboration of the present work, Duminil-Copin, Raoufi and Tassion have developed novel techniques based on the theory of randomized algorithms and have proved new sharp threshold results, see~\cite{duminil2017sharp,duminil2017exponential}. This method has proved robust enough to work in a variety of settings: discrete and continuous (the Ising model and Voronoi percolation), with dependent models (such as FK-percolation) and in any dimension. It seems worthwhile to note that the present model resists direct application. At present, we see at least two obstacles: first of all the influences that arise in our setting are not the same as those of~\cite{duminil2017sharp,duminil2017exponential} (more precisely, the influences studied by Duminil-Copin, Raoufi and Tassion can be expressed as covariances while ours cannot exactly, see Remark~\ref{r.infl_with_cond}). Secondly, right now it is not obvious for us whether or not our measures are monotonic.

\item Another related work whose strategy is closer to the present paper is~\cite{rodriguez20150}, where the author studies similar questions for the $d$-dimensional discrete (massive) Gaussian free field. Some elements of said work apply to general Gaussian fields. More precisely, following the proof of  Proposition 2.2 of~\cite{rodriguez20150}, one can express the derivative probability with respect to the threshold $p$ as a sum of covariances, which seems promising, especially in view of~\cite{duminil2017sharp,duminil2017exponential}. However, each covariance is weighted with a sum of coefficients of the inverse covariance matrix of the discretized field and at present we do not know how to deal with these sums. In~\cite{rodriguez20150}, these coefficients are very simple because we are dealing with the Gaussian free field.

\item The idea to use a KKL inequality to compute the critical point of a planar percolation model comes from~\cite{bollobas2006short,bollobas2006sharp,bollobas2006critical}. See also\cite{beffara2012self,duminil2016new} where such an inequality is used to study FK percolation. In~\cite{beffara2012self,rodriguez20150,duminil2016new}, the authors use a KKL inequality for monotonic measures proved by Graham and Grimmett~\cite{graham2006influence,graham2011sharp}. The same obstacles as in Item~1 above prevented us to use this KKL inequality.

\ei

\paragraph{A note on vocabulary.} We end the first section by a remark on vocabulary on positive definite matrices and functions.

\begin{remark}\label{r.pos_def}
In all the paper, we are going to deal with positive definite functions and matrices. The convention seems to be that, on the one hand positive definite matrices are invertible whereas semi-positive definite matrices are not necessarily. On the other hand, a function $g$ is said strictly positive definite if for any $n \in \Z_{>0}$ and any $x_1, \cdots, x_n$, $(g(x_i-x_j))_{1 \leq i,j \leq n}$ is a positive definite matrix while $g$ is said positive definite if the matrices $(g(x_i-x_j))_{1 \leq i,j \leq n}$ are semi-positive definite. We will follow these conventions and hope this remark will clear up any ambiguities.
\end{remark}

\paragraph{Acknowledgments:}$ $\\
We are grateful to Christophe Garban and Damien Gayet for many fruitful discussions and for their advice in the organization of the manuscript. We would also like to thank Vincent Beffara for his helpful comments. Moreover, we are thankful to Thomas Letendre for pointing out useful references about Gaussian fields and for his help regarding Fourier techniques. Finally, we would like to thank the anonymous referees for their careful reading of the manuscript and their helpful suggestions.

\section{Proof strategy and intermediate results}\label{s.strategy}

In this section, we explain the global strategy of the proof of our main result Theorem~\ref{t.main}. Since the case $p \leq 0$ is already known (see the beginning of Section~\ref{s.main_results}), we focus to the case $p>0$. We first discuss briefly some aspects of the analogous result for Bernoulli percolation: Kesten's theorem. Next, we give an informal explanation of our proof and state the main intermediate results. More precisely, we explain the discretization procedure used in our proof in Subsection~\ref{ss.discretization}. Then, we state the main intermediate results at the discrete level in Subsection~\ref{ss.discrete}, and in Subsection~\ref{ss.from_discrete_to_cont} we explain how to go from the discrete to the continuum.

\subsection{Some ingredients for the proof of Kesten's theorem}\label{ss.kesten}

Several proofs of Kesten's theorem (Theorem~\ref{t.kesten}) are known (see~\cite{grimmett2010probability,bollobas2006percolation}). Let us fix a parameter $p\in[0,1]$ and consider Bernoulli percolation on $\Z^2$ with parameter $p$. As explained before, we are mainly interested in the proof of existence of an unbounded component, i.e. when $p>1/2$. One possible proof of Kesten's theorem uses the following ingredients, that we will try to adapt to our setting. See for instance Section~3.4 of~\cite{bollobas2006percolation}, Section~5.8 of~\cite{grimmett1999percolation} or Section~3.4 of~\cite{garban2014noise} where it is explained how one can combine these ingredients.
\bi 
\item \textit{A box-crossing criterion:} For all $\rho_1,\rho_2>0$, let $\cross^{\text{perco}}_p(\rho_1,\rho_2)$ denote the event that there is a continuous path of black edges that crosses the rectangle $[0,\rho_1] \times [0,\rho_2]$ from left to right, and assume that:
\[
\sum_{k \in \N} \Pro \left[ \neg\cross^{\text{perco}}_p(2^{k+1},2^k) \right] < +\infty \, .
\]
Then, a.s. there exists an infinite black component.
\item \textit{The RSW (Russo-Seymour-Welsh) theorem}, 
which implies that, for every $\rho > 0$, there exists a constant $c=c(\rho) > 0$ such that, for every $R > 0$, $\Pro \left[ \cross^{\text{perco}}_{1/2}(\rho R,R) \right] \geq c(\rho)$.
\item \textit{The FKG-Harris inequality} (see~\cite{grimmett1999percolation,bollobas2006percolation}) that says that increasing events are positively correlated.
\item \textit{Russo's differential formula} (\cite{grimmett1999percolation,bollobas2006percolation}): Let $n \in \N_+$, $\Pro_p^n := \left( p\delta_1+(1-p)\delta_{0} \right)^{\otimes n}$ and $A \subseteq \lbrace 0,1 \rbrace^n$. For every $i \in \lbrace 1, \cdots, n \rbrace$, let $\Infl_i^p(A)$ denote the influence of $i$ on $A$ at the parameter $p$, which is the probability that changing the value of the coordinate $i$ modifies $\un_A$. If $A$ is increasing, then we have the following differential formula:
\[
\frac{d}{dp} \Pro_p^n \left[ A \right] = \sum_{i=1}^n \Infl_i^p(A) \, .
\]
\item \textit{A KKL (Kahn-Kalai-Linial) theorem} (see for instance Theorems~1.16 and~3.4 of~\cite{garban2014noise}): The sum of influences can be estimated thanks to the celebrated KKL theorem. Here, we present the version of the KKL theorem that implies that, if all the influences are small, then the sum of the influences is large. A qualitative version of this principle was proved by Russo in~\cite{russo1982approximate}. The KKL theorem, proved in~\cite{kahn1988influence} for $p=1/2$ and generalized in~\cite{bourgain1992influence,talagrand1994russo} to every $p$, is a quantitative version of~\cite{russo1982approximate}. Let $\Pro_p^n$ be as above. There exists an absolute constant $c > 0$ such that, for every $p \in [0,1]$, every $n \in \N_+$ and every $A \subseteq \lbrace 0,1 \rbrace^n$, we have:
\[
\sum_{i=1}^n \Infl_i^p(A) \geq c \, \Pro_p^n \left[ A \right] \cdot (1-\Pro_p^n \left[ A \right]) \cdot \log \left( \frac{1}{\max_i \Infl_i^p(A)} \right) \, .
\]
The idea to use a KKL theorem to prove Kesten's theorem comes from~\cite{bollobas2006short}.
\ei


Our global strategy to prove Theorem~\ref{t.main} will be based on similar ingredients and on a discretization procedure used in~\cite{bg_16,bm_17} (together with a sprinkling argument). Some of these ingredients are already known, the others will be proved in our paper. We list them in the remaining subsections of Section~\ref{s.strategy}, and in Section~\ref{ss.mainpf} we will explain how we can combine all these ingredients to prove Theorem~\ref{t.main}.\\

Since most of our intermediate results work in a much wider setting, we first state the conditions on the planar Gaussian field $f$ under which we work.

\subsection{Conditions on the Gaussian fields}\label{ss.cond}

First, we state Condition~\ref{a.super-std} \textbf{that we will assume during all the work}:
\begin{condition}\label{a.super-std}
The field $f$ is non-degenerate (i.e. for any pairwise distinct $x_1,\dots,x_k\in\R^2$, the covariance matrix of $(f(x_1),\cdots,f(x_k))$ is invertible), centered, normalized (i.e. $\var(f(x))=1$ for every $x$), continuous, and stationary. In particular, there exists a strictly positive definite continuous function $\kappa \, : \, \R^2 \rightarrow [-1,1]$ such that $K(x,y) := \E \left[ f(x) f(y) \right] = \kappa(y-x)$ and $\kappa(0)=1$.
\end{condition}

Depending on the intermediate results we prove, we will also need to assume some of the following additional conditions:

\begin{condition}[Useful to apply percolation arguments.]\label{a.std}
The field $f$ is positively correlated, invariant by $\frac{\pi}{2}$-rotation, and reflection through the horizontal axis.
\end{condition}

\begin{condition}[Useful to have quasi-indepence.] Depends on a parameter $\alpha>0$.]\label{a.pol_decay}
There exists $C<+\infty$ such that for each $x\in\R^2$, $|\kappa(x)|\leq C|x|^{-\alpha}$.
\end{condition}

\begin{condition}[Technical conditions to have quasi-independence, see~\cite{rv_rsw}.]\label{a.decay} The function $\kappa$ is $C^8$ and for each $\beta\in\N^2$ with $\beta_1+\beta_2\leq 2$, $\partial^\beta\kappa(x)\underset{|x|\rightarrow +\infty}{\longrightarrow}0$.
\end{condition}

\begin{condition}[Useful to do Fourier calculations on the correlation function. Depends on a parameter $\alpha>0$.]\label{a.fourier}
The Fourier transform of $\kappa$ takes only positive values. Moreover, $\kappa$ is $C^3$ and there exists $C<+\infty$ such that for every $\beta \in \N^2$ with $\beta_1 + \beta_2 \leq 3$, we have:
\[
|\partial^\beta \kappa(x)| \leq C |x|^{-\alpha} \, .
\]
\end{condition}

We will often suppose regularity conditions on $f$ and $\kappa$. It will be interesting to have the following in mind (see for instance Appendices A.3 and A.9 of~\cite{nazarov2015asymptotic}):
\begin{lem}\label{l.reg}
Assume that $f$ satisfies Condition~\ref{a.super-std}. Let $k \in \N$. If $\kappa$ is $C^{2(k+1)}$, then a.s. $f$ is $C^k$. Conversely, if a.s. $f$ is $C^k$, then $\kappa$ is $C^{2k}$ and for every multi-indices $\beta,\gamma \in \N^2$ such that $\beta_1+\beta_2, \leq k$ and $\gamma_1+\gamma_2 \leq k$, we have:
\[
\cov \left( \partial^\beta f(x), \partial^\gamma f(y) \right) = \E \left[ \partial^\beta f(x) \partial^\gamma f(y) \right] = (-1)^{|\gamma|} \partial^{\beta+\gamma}\kappa (x-y) \, .
\]
\end{lem}

It is easy to check that the conditions are all satisfied by the Bargmann-Fock field (see Lemma~\ref{l.Bochnereasy_strict} for the non-degeneracy condition). In particular, Conditions~\ref{a.decay} and~\ref{a.fourier} hold for any $\alpha > 0$. Also, the Bargmann-Fock field is a.s. analytic and its covariance is analytic.\\

An other example of Gaussian fields in the plane that satisfy the above conditions (with the parameter $\alpha$ which depends on the parameter $n$) is the field with correlation function:
\begin{equation}\label{e.1over1+xsquared}
\kappa \, : \, x=(x_1,x_2) \in \R^2  \mapsto \left( \frac{1}{(1+x_1^2)(1+x_2^2)} \right)^n \, ,
\end{equation}
where $n \in \Z_{>0}$. This is indeed a strictly positive definite function by the following lemma.

\begin{lem}\label{l.Bochnereasy_strict}
The Fourier transform of a continuous and integrable function $\R^2 \rightarrow \R_+$ which is not $0$ is strictly positive definite. In particular, the Gaussian function $x \mapsto \exp(-\frac{1}{2}|x|^2)$ and the function~\eqref{e.1over1+xsquared} (for any $n \in \Z_{>0}$) are strictly positive definite.
\end{lem}
\begin{proof}
This is a direct consequence of Theorem~3 of Chapter~13 of~\cite{cheney2009course} (which is the strictly positive definite version of the easy part of Bochner theorem). We can apply this to the Bargmann-Fock field and to~\eqref{e.1over1+xsquared} since the Fourier transform of a Gaussian function is still a Gaussian function and since the Fourier transform of $x \mapsto \frac{1}{(1+x_1^2)(1+x_2^2)}$ is $\xi \mapsto cst \, \exp(-(|\xi_1|+|\xi_2|))$, hence $x \in \R^2  \mapsto \left( \frac{1}{(1+x_1^2)(1+x_2^2)} \right)^n$ is the Fourier transform of the function $\xi \mapsto cst \, \exp(-(|\xi_1|+|\xi_2|))$ convoluted $n$ times, see Paragraph~1.2.3 of~\cite{rud_fou}.
\end{proof}

If one wants to consider a large family of examples of planar Gaussian fields, one can consider a function $\rho \, : \, \R^2 \rightarrow \R_+$ sufficiently smooth, that is not $0$ and such that $\rho$ and its derivatives decay sufficiently fast. One can note that $\kappa = \rho*\rho$ has the same properties and is strictly positive definite by Lemma~\ref{l.Bochnereasy_strict} since $\hat{\kappa}=(\hat{\rho})^2$. Moreover, if $\rho$ has sufficiently many symmetries, then the Gaussian field with covariance $\kappa$ above satisfies Conditions from~\ref{a.super-std} to~\ref{a.fourier}. Now, if $W$ is a two-dimensional white noise, that is, the free field associated to the Hilbert space $L^2(\R^2)$ (see Definition 2.5 of \cite{shef_gff}), then (if $\rho$ is even)
\[
f = \rho*W
\]
defines a Gaussian field on $\R^2$ with covariance $\kappa$.\\

We now list the main intermediate results of our proof.

\subsection{A box-crossing criterion}\label{ss.box-cross}

As in the strategy of the proof of Kesten's theorem presented in Subsection~\ref{ss.kesten}, our goal will be to prove a box-crossing criterion. We start by introducing the following notation.
\begin{notation}\label{n.cross}
For each $\rho_1,\rho_2>0$ and each $p\in\R$, we write $\cross_p(\rho_1,\rho_2)$ for the event that there is a continuous path in $\calD_p\cap [0,\rho_1]\times[0,\rho_2]$ joining the left side of $[0,\rho_1]\times[0,\rho_2]$ to its right side.
\end{notation}
In Subsection \ref{ss.criterion}, we will prove the following proposition.
\begin{lem}\label{l.criterion}
Assume that $f$ satisfies Condition~\ref{a.super-std} and that $\kappa$ is invariant by $\frac{\pi}{2}$-rotations. Let $p>0$ and assume that:
\begin{equation}\label{e.criterion}
\sum_{k \in \N} \Pro \left[ \neg\cross_p(2^{k+1},2^k) \right] < +\infty \, .
\end{equation}
Then, a.s. there exists a unique unbounded component in $\calD_p$.
\end{lem}

Thus, our goal turns to prove that, if $p > 0$, then $\Pro \left[ \cross_p(2R,R) \right]$ goes to $1$ as $R$ goes to $+\infty$, sufficiently fast so that the above sum is finite. In order to prove such a result, we will show a Russo-type formula and a KKL-type theorem for discrete Gaussian fields. To apply such result, we first need to discretize our model, as it was done in~\cite{bg_16}.

\subsection{A discretization procedure and a discrete phase transition theorem for more general fields}\label{ss.discretization}

We consider the following discrete percolation model: let $\mathcal{T}=(\mathcal{V},\mathcal{E})$ be the face-centered square lattice defined in Figure~\ref{f.face-centered}. Note that the exact choice of lattice is not essential in most of our arguments. We mostly use the fact that it is a periodic triangulation with nice symmetries. However, we do need a little more information to apply Theorem~\ref{t.epsRSW} (see the beginning of Section 3 of~\cite{rv_rsw}) and in Section~\ref{s.sqrt} of the present work we use the fact that the sites of $\calT$ are the vertices of a rotated and rescaled $\Z^2$-lattice. We will denote by $\calT^\eps=(\calV^\eps,\calE^\eps)$ this lattice scaled by a factor $\eps$. Given a realization of our Gaussian field $f$ and some $\eps > 0$, we color the plane as follows: For each  $x \in \R^2$, if $x \in \mathcal{V}^\eps$ and $f(x) \geq -p$ or if $x$ belongs to an edge of $\calE^\eps$ whose two extremities $y_1,y_2$ satisfy $f(y_1) \geq -p$ and $f(y_2) \geq -p$, then $x$ is colored black. Otherwise, $x$ is colored white. In other words, we study a \textbf{correlated site percolation model} on $\cal T^\eps$.

\begin{figure}[!h]
\begin{center}
\includegraphics[scale=0.45]{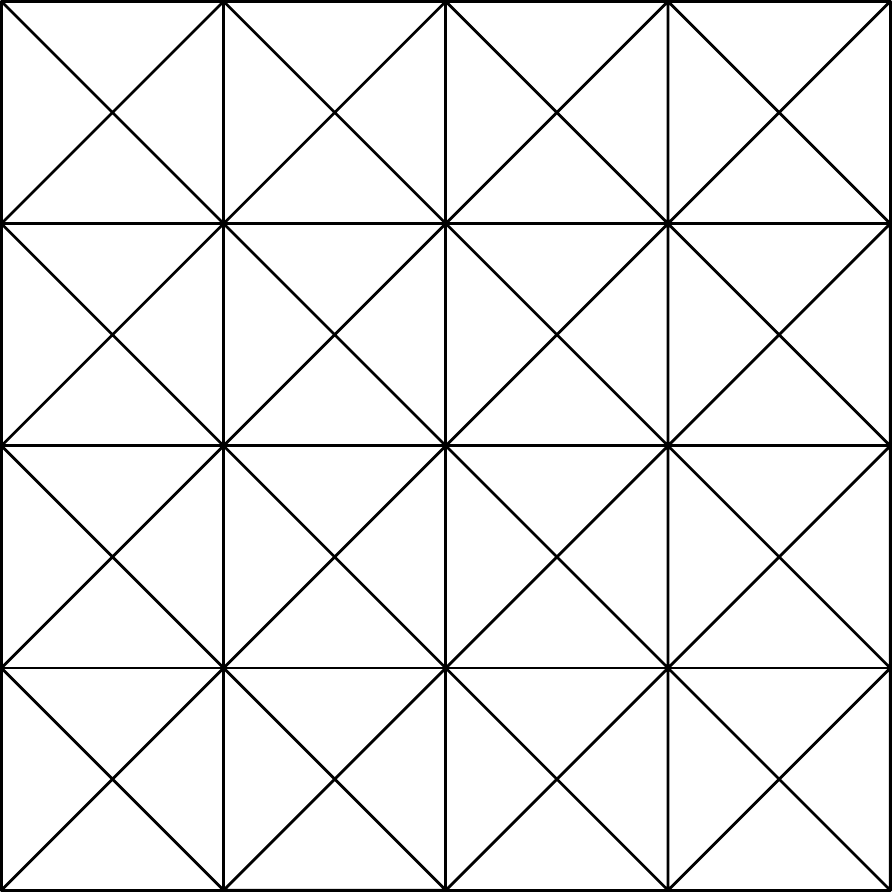}
\end{center}
\caption{The face-centered square lattice (the vertices are the points of $\Z^2$ and the centers of the squares of the $\Z^2$-lattice).}\label{f.face-centered}
\end{figure}

We will use the following notation analogous to Notation~\ref{n.cross}.
\begin{notation}\label{n.crosseps}
Let $\eps>0$, $p\in\R$, and consider the above discrete percolation model. For every $\rho_1,\rho_2 > 0$, write $\cross_p^\eps(\rho_1,\rho_2)$ for the event that there is a continuous black path in $[0,\rho_1]\times[0,\rho_2]$ joining the left side of $[0,\rho_1]\times[0,\rho_2]$ to its right side.
\end{notation}

As explained in Subsection~\ref{ss.box-cross}, we want to estimate the quantities $\Pro \left[ \cross_p(2R,R) \right]$ for $p>0$. However, the tools that we develop in our paper are suitable to study the quantities $\Pro \left[ \cross_p^\eps(2R,R) \right]$. We will thus have to choose $\eps$ so that, on the one hand, we can find nice estimates at the discrete level and, on the other hand, the discrete field is a good approximation of the continuous field. As a reading guide for the global strategy, we list below what are the constraints on $\eps$ for our intermediate results to hold. See also Figure~\ref{f.constraints}. We write them for the Bargmann-Fock field. Actually, most of the intermediate results including Item~2 below hold for more general fields, see the following subsections for more details.

\bi 
\item[1.] We want to establish a lower bound for $\Pro \left[ \cross_p^\eps(2R,R) \right]$ that goes (sufficiently fast for us) to $1$ as $R$ goes to $+\infty$. Our techniques work as long as the number of vertices is not too large (see in particular Subsection~\ref{sss.square_root}). They yield such a bound provided that there exists $\delta>0$ such that $(R,\eps)$ satisfy the condition:
\begin{equation}\label{e.constraint1}
\eps \geq \log^{-(1/2-\delta)}(R)\, .
\end{equation}
See Subsubsection~\ref{sss.combining}.

\item[2.] Then, we will have to estimate how much $\cross_p^\eps(2R,R)$ approximates well $\cross_p(2R,R)$. At this point, it seems natural to use the quantitative approximation results from~\cite{bg_16,bm_17}. In particular, as explained briefly in Subsection~\ref{ss.from_discrete_to_cont}, it seems that results from~\cite{bm_17} (based on discretization schemes that generalize the methods in~\cite{bg_16}) imply that the event $\cross_p^\eps(2R,R)$ approximates $\cross_p(2R,R)$ well if:
\begin{equation}\label{e.constraint2}
\eps \leq R^{-(1+\delta)} \, ,
\end{equation}
for some $\delta > 0$. Unfortunately, this constraint is not compatible with the constraint~\eqref{e.constraint1}. For this reason, we will instead use a sprinkling discretization procedure. More precisely, we will see in Proposition~\ref{p.sprinkling} that $\cross^\eps_{p/2}(2R,R)$ implies $\cross_p(2R,R)$ with high probability if $(R,\eps)$ satisfy the following condition:
\begin{equation}\label{e.constraint3}
\eps \leq \log^{-(1/4+\delta)}(R)
\end{equation}
for some $\delta > 0$. This time, the constraint combines very well with the constraint~\eqref{e.constraint1} and we will be able to conclude.
\ei

\begin{figure}[!h]
\begin{center}
\includegraphics[width=\textwidth]{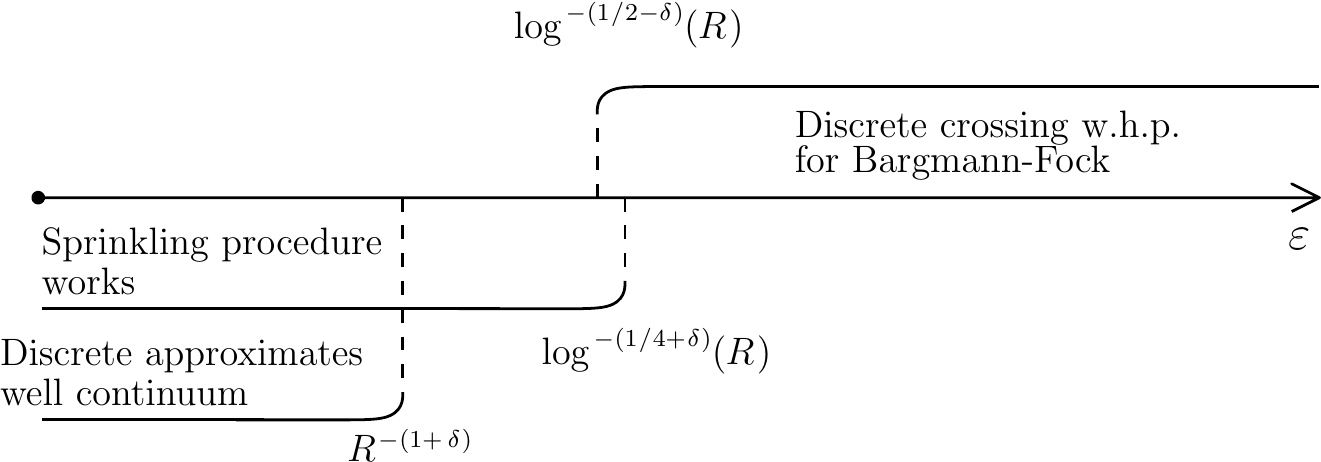}
\end{center}
\caption{The different constraints on $(R,\eps)$.}\label{f.constraints}
\end{figure}

If we work with the Bargmann-Fock field and if we choose for instance $\eps=\eps(R)=\log^{-1/3}(R)$ then, as shown in Figure~\ref{f.constraints}, we obtain that $\cross^\eps_p(2R,R)$ holds with high probability and that the sprinkling discretization procedure works. As explained in Section~\ref{ss.mainpf}, we will thus obtain that $\sum_{k \in \N} \Pro \left[ \neg\cross_p(2^{k+1},2^k) \right] < +\infty$ and conclude thanks to Lemma~\ref{l.criterion}. See also Figure~\ref{f.strat}.\\

\begin{figure}[!h]
\begin{center}
\includegraphics[scale=0.50]{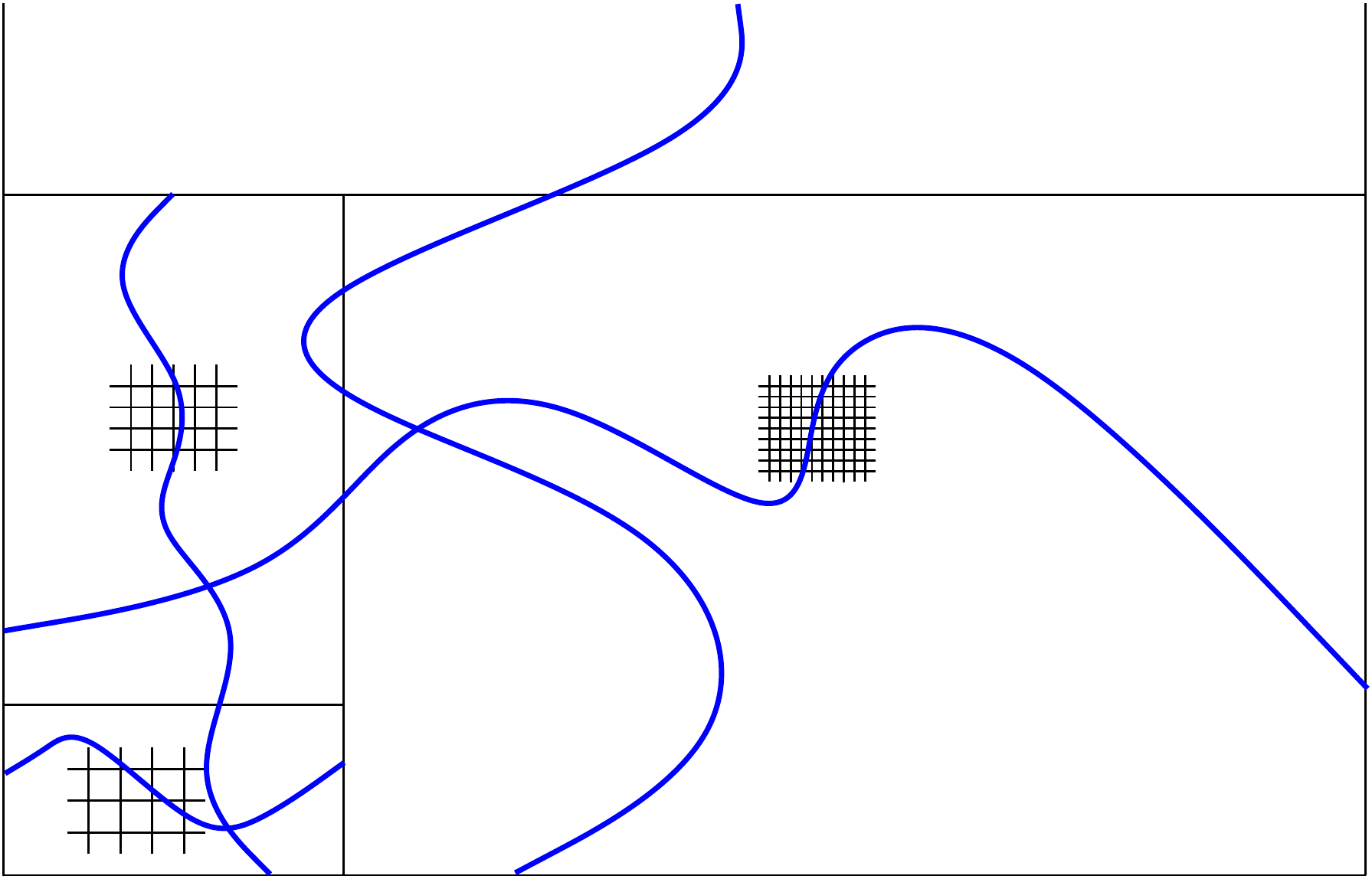}
\end{center}
\caption{A right choice of $\eps$ is $\eps=\log^{-1/3}(R)$ with whom we will be able to prove that $\Pro \left[ \cross(2R,R) \right]$ goes sufficiently fast to $1$ so that the box-crossing criterion of Lemma~\ref{l.criterion} is satisfied. As we will see in the proof of this lemma, this criterion implies that, for a well-chosen collection of rectangles $(Q_k)_k$ of typical length $R_k$, a.s. all the $Q_k$'s for $k$ sufficiently large are crossed, which implies that there exists an infinite path. The idea is that, for each $k$, we study the probability that it is crossed by comparing the continuous event with the discrete event with mesh $\eps =\log^{-1/3}(R_k)$.}\label{f.strat}
\end{figure}

Note that Item~1 above implies in particular that, if we work with the Bargmann-Fock field and if we fix $\eps \in ]0,1],$\footnote{In this paper, we only look at the case $\eps \in ]0,1]$ though a lot of results could probably be extended to any $\eps > 0$.} then for $p>0$, $\Pro \left[ \cross^\eps_p(2R,R) \right]$ goes to $1$ (sufficiently fast for us) as $R$ goes to $+\infty$. We will actually obtain such a result for more general fields, which will enable us to prove that a discrete box-crossing criterion analogous to Lemma~\ref{l.criterion} is satisfied and deduce the following in Section~\ref{ss.mainpf}:


\begin{thm}\label{t.epsKesten}
Suppose that $f$ satisfies Conditions~\ref{a.super-std} and~\ref{a.std} and Condition~\ref{a.fourier} for some $\alpha > 5$ (thus in particular Condition~\ref{a.pol_decay} is satisfied for some $\alpha > 5$). Then, for each $\eps \in ]0,1]$, the critical threshold of the discrete percolation model on $\calT^\eps$ defined in the present subsection is $0$. More precisely: the probability that there is an unbounded black connected component is $1$ if $p > 0$, and $0$ otherwise. Moreover, if it exists, such a component is a.s. unique.\\
In particular, this result holds when $f$ is the Bargmann-Fock field or the centered Gaussian field with covariance given by~\eqref{e.1over1+xsquared} with $n \geq 3$.
\end{thm}

As in the continuous setting, the case $p \leq 0$ of Theorem~\ref{t.epsKesten} goes back to~\cite{bg_16}, at least for $\alpha$ large enough. In~\cite{rv_rsw}, we have optimized this result and obtain the following:
\begin{thm}[\cite{bg_16} for $\alpha$ sufficiently large,\cite{rv_rsw}]\footnote{More precisely, this is Proposition~$\harrisdiscrete$ of~\cite{rv_rsw}. Moreover, this can be extracted from the proof of Theorem~5.7 of~\cite{bg_16} for $\alpha$ sufficiently large and with slightly different assumptions on the differentiability and the non-degeneracy of $\kappa$.}\label{t.cor_bg_eps} Assume that $f$ satisfies Conditions~\ref{a.super-std},~\ref{a.std},~\ref{a.decay} as well as Condition~\ref{a.pol_decay} for some $\alpha > 4$. Let $\eps \in ]0,1]$ and consider the discrete percolation model on $\calT^\eps$ defined in the present subsection with parameter $p=0$. Then, a.s. there is no infinite connected black component.
\end{thm}

Again as in the continuous case, this result heavily relies on a RSW estimate, which we state below (see Theorem~\ref{t.epsRSW}). This estimate is a uniform lower bound on the crossing probability of a quad scaled by $R$ on a lattice with mesh $\eps$. The proof of such an estimate goes back to~\cite{bg_16} (for $\alpha > 16$) and was later optimized to $\alpha > 8$ in~\cite{bm_17}. The key property in this estimate is that the lower bound of the crossing properties does not depend on the choice of the mesh $\eps$. To obtain such a result, the authors of~\cite{bg_16} had to impose some conditons on $(R,\eps)$. For instance, if we consider the Bargmann-Fock field, the constraint was $\eps \geq C\exp(-cR^2)$ where $C<+\infty$ and $c>0$ are fixed. Actually, it seems likely that, one could deduce a discrete RSW estimate with no constraint on $(R,\eps)$ by using their quantitative approximation results. In~\cite{rv_rsw}, we prove such a discrete RSW estimate with no constraint on $(R,\eps)$ and without using quantitative discretization estimates, but rather by using new quasi-independence results. Note that to prove our main result Theorem~\ref{t.main}, it would not have been a problem for us to rather use the result by Beffara and Gayet with constraints on $(R,\eps)$ since, as explained above, we will use results from the discrete model with $\eps=\eps(R)=\log^{-1/3}(R)$ - which of course satisfies the condition $\eps \geq C\exp(-cR^2)$. We have the following:
\begin{thm}[\cite{bg_16} for $\alpha > 16$,~\cite{bm_17} for $\alpha > 8$,\cite{rv_rsw}]\footnote{More precisely, this is Proposition~$\tassiondiscrete$ of~\cite{rv_rsw}. Moreover, this is Theorem~2.2 from~\cite{bg_16} combined with the results of their Section~4 (resp. this is Appendix~C of~\cite{bm_17}) for $\alpha > 16$ (resp. $\alpha > 8$), with some constraints on $(R,\eps)$, and with slightly different assumptions on the differentiability and the non-degeneracy of $\kappa$.}\label{t.epsRSW} Assume that $f$ satisfies Conditions~\ref{a.super-std},~\ref{a.std},~\ref{a.decay} as well as Condition~\ref{a.pol_decay} for some $\alpha > 4$. For every $\rho \in ]0,+\infty[$, there exists a constant $c=c(\kappa,\rho)>0$ such that the following holds: let $\eps \in ]0,1]$, and consider the discrete percolation model on $\calT^\eps$ defined in the present subsection with parameter $p=0$. Then, for every $R \in ]0,+\infty[$ we have:
\[
\Pro \left[ \cross^\eps_0(\rho R,R) \right] \geq c \, .
\]
\end{thm}

What remains to prove in order to show Theorem~\ref{t.epsKesten} is that, if $p>0$, a.s. there is a unique infinite black component. Exactly as in the continuum, our goal will be to show that a box-crossing criterion is satisfied. The following lemma is proved in Subsection~\ref{ss.criterion}.

\begin{lem}\label{l.criterion_eps}
Assume that $f$ satisfies Condition~\ref{a.super-std} and that $\kappa$ is invariant by $\frac{\pi}{2}$-rotations. Let $\eps > 0$, let $p \in \R$ and suppose that:
\begin{equation}\label{e.criterion_eps}
\sum_{k \in \N} \Pro \left[ \neg \cross_p^\eps(2^{k+1},2^k) \right] < +\infty \, .
\end{equation}
Then, a.s. there exists a unique unbounded black component in the discrete percolation model defined in the present subsection.
\end{lem}

\subsection{Sharp threshold at the discrete level}\label{ss.discrete}

We list the intermediate results at the discrete level. Among all the following results (and actually among all the intermediate results of the paper) the only result specific to the Bargmann-Fock field is the second result of Proposition~\ref{p.square_root_estimate}. All the others work in a quite general setting.

\subsubsection{The FKG inequality for Gaussian vectors}

The FKG inequality is a crucial tool to apply percolation arguments. We say that a Borel subset $A\subseteq\R^n$ is \textbf{increasing} if for each $x \in A$ and $y\in\R^n$ such that $y_i \geq x_i$ for every $i \in \lbrace 1, \cdots, n \rbrace$, we have $y \in A$. We say that $A$ is decreasing if the complement of $A$ is increasing. The FKG inequality for Gaussian vectors was proved by Pitt and can be stated as follows.
\begin{thm}[\cite{pitt1982positively}]\label{t.FKG}
Let $X = (X_1, \cdots, X_n)$ be a $n$-dimensional Gaussian vector whose correlation matrix has non-negative entries. Then, for every increasing Borel subsets $A,B \subseteq \R^n$, we have:
\[
\Pro \left[ X \in A \cap B \right] \geq \Pro \left[ X \in A \right] \cdot \Pro \left[ X \in B \right] \, .
\]
\end{thm}

This result is the reason why we work with positively correlated Gaussian fields. Indeed, the FKG inequality is a crucial ingredient in the proof of RSW-type results. Note however that the very recent~\cite{bg_17} proves a box-crossing property without this inequality, albeit only in a discrete setting.

\subsubsection{A differential formula}\label{sss.russo}

As in the case of Bernoulli percolation, we need to introduce a notion of influence. This notion of influence is inspired by the \textbf{geometric influences} studied by Keller, Mossel and Sen in the case of product spaces~\cite{keller2012geometric,keller2014geometric}. (For more about the relations between the geometric influences and the influences of Definition~\ref{d.infl} below - which are roughly the same in the case of product measures - see Subsection~\ref{ss.KMSnonproduct}.)

\begin{defi}\label{d.infl}
Let $\mu$ be a finite Borel measure on $\R^n$, let $v \in \R^n$, and let $A$ be a Borel subset of $\R^n$.  The influence of $v$ on $A$ under $\mu$ is:
\[
I_{v,\mu}(A) := \underset{r \downarrow 0}{\liminf} \, \frac{\mu \left( A + [-r,r] \, v \right) - \mu \left( A \right)}{r} \in [0,+\infty] \, .
\]
Write $\left( e_1, \cdots, e_n \right)$ for the canonical basis of $\R^n$. We will use the following simplified notations:
\[
I_{i,\mu}(A) := I_{e_i,\mu}(A) \, .
\]
\end{defi}

The events we are interested in are ``threshold events'' and the measures we are interested in are Gaussian distributions: Let $X=(X_1, \cdots, X_n)$ be a $n$-dimensional non-degenerate centered Gaussian vector, write $\mu_X$ for the law of $X$ and, for every $\overrightarrow{p} = (p_1, \cdots, p_n) \in \R^n$ and every $i \in \lbrace 1, \cdots, n \rbrace$, write:
\[
\omega^{\overrightarrow{p}}_i := \un_{\lbrace X_i \geq -p_i \rbrace} \, .
\] 
This defines a random variable $\omega^{\overrightarrow{p}}$ with values in $\lbrace 0,1 \rbrace^n$. If $B \subseteq \lbrace 0,1 \rbrace^n$, we call the event $\lbrace \omega^{\overrightarrow{p}} \in B \rbrace$ a \textbf{threshold event}. For every $i \in \lbrace 1, \cdots, n \rbrace$, we let $\Piv_i^{\overrightarrow{p}}(B)$ denote the event that changing the value of the bit $i$ in $\omega^{\overrightarrow{p}}$ modifies $\un_B(\omega^{\overrightarrow{p}})$. In other words,
\[
\Piv^{\overrightarrow{p}}_i(B)=\lbrace (\omega^{\overrightarrow{p}}_1,\dots,\omega^{\overrightarrow{p}}_{i-1},0,\omega^{\overrightarrow{p}}_{i+1},\dots,\omega^{\overrightarrow{p}}_n)\in B\rbrace\bigtriangleup\lbrace (\omega^{\overrightarrow{p}}_1,\dots,\omega^{\overrightarrow{p}}_{i-1},1,\omega^{\overrightarrow{p}}_{i+1},\dots,\omega^{\overrightarrow{p}}_n)\in B\rbrace \, ,
\]
where $E \bigtriangleup F := (E\setminus F) \cup (F \setminus E)$. Such an event is called a \textbf{pivotal event}. We say that a subset $B\subseteq \lbrace 0,1\rbrace^n$ is \textbf{increasing} if for every $\omega \in B$ and $\omega'\in \lbrace 0,1\rbrace^n$, the fact that $\omega'_i \geq \omega_i$ for every $i$ implies that $\omega' \in B$. Moreover, if $p \in \R$, we write $\mathbf{p}=(p,\cdots,p) \in \R^n$ and we use the following notations:
\[
\omega^p := \omega^{\mathbf{p}} \: \text{  and  } \: \Piv_i^p(B) := \Piv_i^{\mathbf{p}}(B) \, .
\]

\begin{prop}\label{p.russo}
Assume that $X$ is a $n$-dimensional non-degenerate centered Gaussian vector and let $\Sigma$ be its covariance matrix. Let $B$ be an increasing subset of $\lbrace 0,1 \rbrace^n$. Then:
\[
\frac{\partial\prob\left[\omega^{\vp}\in B\right]}{\partial p_i}=I_{i,\mu_X}(\omega^{\vp}\in B)=\prob\left[\piv_i^{\vp}(B) \cond X_i=-p_i\right]\frac{1}{\sqrt{2\pi\Sigma_{i,i}}}\exp \left(-\frac{1}{2\Sigma_{i,i}}p_i^2\right) \, .
\]
In particular, if $p \in \R$, then:
\[
\frac{d \, \Pro \left[ \omega^p \in B \right]}{dp} = \sum_{i=1}^n I_{i,\mu_X}(\omega^p \in B) = \sum_{i=1}^n \Pro \left[ \Piv_i^p(B) \cond X_i = -p \right]\frac{1}{\sqrt{2\pi\Sigma_{i,i}}}\exp \left(-\frac{1}{2\Sigma_{i,i}}p^2\right) \, .
\] 
\end{prop}

\begin{remark}\label{r.infl_with_cond}
With the same hypotheses as in Proposition~\ref{p.russo}: For $\epsilon \in \{ 0,1 \}$, let $B^i_\epsilon = \{ \omega \in \{ 0,1 \}^n \, : \, (\omega_1, \cdots, \omega_{i-1}, \epsilon, \omega_{i+1}, \cdots, \omega_n) \in B \}$. Note that we have: \[
\prob\left[\piv_i^{\vp}(B) \cond X_i=-p_i\right] = \prob \left[ \omega^{\vp} \in B_1^i \cond X_i = -p_i \right] -  \prob \left[ \omega^{\vp} \in B_0^i \cond X_i = -p_i \right] \, .
\]
Hence the form of our influences is close to the covariance between $\un_{\{ \omega_i^{\vp} = 1 \}}$ and $\un_{\{ \omega^{\vp} \in B \}}$ but is not exactly a covariance. This is important to have in mind if one wants to try to apply techniques from~\cite{duminil2017sharp,duminil2017exponential} in the future.
\end{remark}

Since the proof of Proposition~\ref{p.russo} is rather short, we include this here. The reader essentially interested in the strategy of proof can skip this in a first reading.
\begin{proof}[Proof of Proposition \ref{p.russo}]
Without loss of generality, we assume that $i=n$. For any $C \subseteq \{ 0,1 \}^n$, let $C^{\vp}\subseteq\R^n$ be the preimage of $C$ by the map $x \in \R^n \mapsto (\un_{x_i \geq -p_i})_{i \in \{ 1, \cdots, n \}} \in \{ 0,1 \}^n$. Let $e_n=(0,\dots,0,1)\in\R^n$, $h>0$, let $B \subseteq \{ 0,1 \}^n$ be an increasing event, and let $\tilde{X}=(X_i)_{1\leq i\leq n-1}$. Then:
\[
    \prob\left[\omega^{\vp+he_n}\in B\right]=\prob\left[X\in B^{\vp+he_n}\right]=\prob\left[X+he_n\in B^{\vp}\right]=\prob\left[(\tilde{X},X_n+h)\in B^{\vp}\right] \, ,
\]
and $\prob\left[\omega^{\vp}\in B\right]=\prob\left[(\tilde{X},X_n)\in B^{\vp}\right]$.\\

Also:
  \begin{align*}
    \prob\left[(\tilde{X},X_n+h)\in 
    B^{\vp}\right]-\prob\left[(\tilde{X},X_n)\in B^{\vp}\right]&=\prob\left[(\tilde{X},X_n+h)\in B^{\vp}, \, (\tilde{X},X_n)\notin B^{\vp}\right]\\
    &=\prob\left[\Piv^{\vp}_n(B),\, X_n\in[-p_n-h,-p_n[\right] \, .
  \end{align*}
  Since $\Sigma$ is positive definite, the Gaussian measure has smooth density with respect to the Lebesgue measure. Taking the difference of the two probabilities and letting $h \downarrow 0$, we get:
  \begin{align*}
    &h^{-1}\left(\prob\left[\omega^{\vp+he_n}\in B\right]-\prob[\omega^{\vp}\in B)]\right)\\
    &=h^{-1}\prob\left[\piv_n^{\vp}(B),\ X_n\in[-p_n-h,-p_n[\right]\\
    &=\frac{1}{h\sqrt{2\pi\Sigma_{n,n}}}\int_{-p_n-h}^{-p_n}\Pro\left[\Piv_n^{\vp}(B) \cond X_n=t \right]\exp \left( -\frac{1}{2\Sigma_{n,n}}t^2 \right)dt\\
        &\underset{h \downarrow 0}{\longrightarrow} \prob\left[\piv_n^{\vp}(B) \cond X_n=-p_n\right]\frac{1}{\sqrt{2\pi\Sigma_{n,n}}} \exp \left( -\frac{1}{2\Sigma_{n,n}}p_n^2 \right) \, .
  \end{align*}
In the last step we use the continuity of the conditional probability of a threshold event with respect to the conditioning value. This is an easy consequence of Proposition 1.2 of \cite{azws}. The calculation for $h<0$ is analogous, hence we are done.
\end{proof}

\subsubsection{A KKL-KMS (Kahn-Kalai-Linial -- Keller-Mossel-Sen) theorem for non-product Gaussian vectors}\label{sss.KMS}

One of the contributions of this paper is the derivation of a KKL theorem for non-product Gaussian vectors, namely Theorem~\ref{t.KMSnonproduct}, based on a similar result for product Gaussian vectors by Keller, Mossel and Sen,~\cite{keller2012geometric}. (This similar result by~\cite{keller2012geometric} is stated in Theorem~\ref{t.KMS} of our paper\footnote{As we will explain, Theorem~\ref{t.KMS} is a simple consequence of Item~$(1)$ of Theorem~1.5 of~\cite{keller2012geometric}.}.) Actually, with our techniques of Section~\ref{s.KMS}, most of the results from~\cite{keller2012geometric} could be extended to non-product Gaussian vectors (and to monotonic events).

\begin{thm}\label{t.KMSnonproduct}
There exists an absolute constant $c > 0$ such that the following holds: Let $n \in \Z_{>0}$, let $\Sigma$ be a $n \times n$ symmetric positive definite matrix\footnote{Remember Remark~\ref{r.pos_def}: this means in particular that $\Sigma$ is non-degenerate.} and let $\mu = \mathcal{N}(0,\Sigma)$. Also, let $\sqrt{\Sigma}$ be a symmetric square root of $\Sigma$ and write $\| \sqrt{\Sigma} \|_{\infty,op}$ for the operator norm of $\sqrt{\Sigma}$ for the infinite norm\footnote{I.e. $\| \sqrt{\Sigma} \|_{\infty,op} = \sup_{x\in \R^n \setminus \{0\}} \frac{|\sqrt{\Sigma}(x) |_\infty}{| x |_\infty}$. Equivalently, $\| \sqrt{\Sigma} \|_{\infty,op} = \underset{i \in \lbrace 1, \cdots, n \rbrace}{\text{\textup{max}}} \sum_{j=1}^n |\sqrt{\Sigma}(i,j)|$.} on $\R^n$. For every $A$ monotonic Borel subset of $\R^n$ we have:

\[
\sum_{i=1}^n I_{i,\mu}(A) \geq c \, \| \sqrt{\Sigma} \|_{\infty,op}^{-1} \, \mu(A) \, (1-\mu(A)) \, \sqrt{\log_+ \left( \frac{1}{\| \sqrt{\Sigma} \|_{\infty,op} \cdot \underset{i \in \lbrace 1, \cdots, n \rbrace}{\text{\textup{max}}} I_{i,\mu}(A)} \right)} \, .
\]
\end{thm}

This theorem holds for a wider class of sets. For instance, it holds for semi-algebraic\footnote{We say that a set $A\subset\R^n$ is semi-algebraic if it belongs to the Boolean algebra generated by sets of the form $P^{-1}\left(]0,+\infty[\right)$ where $P\in\R[X_1,\dots,X_n]$.}, see Appendix A of Chapter 7 of \cite{alelele} or Appendix A of Chapter 2 of \cite{hugogogo}. Since we need it only for monotonic events, we did not try to identify the weakest possible assumptions for the property to  hold. In particular, we have not found any examples of Borel sets for  which the theorem does not hold.\\

A way to understand the constant $\| \sqrt{\Sigma} \|_{\infty,op}$ is to see what happens in the extreme cases:
\bi 
\item If $\Sigma$ is the identity matrix, then $\| \sqrt{\Sigma} \|_{\infty,op}=1$ and the above result corresponds to the product case from~\cite{keller2012geometric}.
\item If $\Sigma_{i,j}=1$ for every $i,j$ (which corresponds to the fact that, if $X \sim \Sigma$, then $X_i=X_j$ for every $i,j$) then $\| \sqrt{\Sigma} \|_{\infty,op}=n$ which is coherent since there is no threshold phenomenon for a single variable.
\ei

The proof of Theorem~\ref{t.KMSnonproduct} is organized as follows: In Subsection~\ref{ss.KMSnonproduct}, we explain how to deduce Theorem~\ref{t.KMSnonproduct} from a sub-linear property of the influences (see Proposition~\ref{p.sublin}) and from the results of~\cite{keller2012geometric}. In Subsection~\ref{ss.monotonic}, we prove the sub-linearity property for monotonic subsets.\\

Remember that we want to prove that $\Pro \left[ \cross_p(2R,R) \right]$ is close to $1$ and that our first step is to prove that it is the case for $\Pro \left[ \cross_p^\eps(2R,R) \right]$. To do so, we use that for $p=0$ this probability is bounded from below (by Theorem~\ref{t.epsRSW}) and we differentiate the probability with respect to $p$ using Proposition~\ref{p.russo}. We then apply Theorem~\ref{t.KMSnonproduct} to the right-hand side so that it is sufficient to prove that the maximum of the corresponding influences is small and that the operator norm for the infinite norm of the correlation matrix of our model is not too large. In the two following subsubsections, we state results in this spirit.

\subsubsection{Polynomial decay of influences}\label{sss.pol}

Thanks to the RSW estimate and quasi-independence results, one can obtain that the probability of ``discrete arm events'' decay polynomially fast, see Subsection~5.3 of~\cite{bg_16} and Proposition~$\polynomially$ of~\cite{rv_rsw}. Such a result together with monotonic arguments imply that we have a polynomial bound on the influences for crossing events. We state this bound here and we prove it in Subsection~\ref{ss.infl_decr}. We first need some notations. Let $\mathcal{V}_R^\eps$ denote the set of vertices $x \in \calV^\eps$ that belong to an edge that intersects $[0,2R] \times [0,R]$, and let $X$ be our Gaussian field $f$ restricted to $\calV^\eps_R$. Thus, $X$ is a finite dimensional Gaussian field. For every $p \in \R$, $x \in \calV^\eps_R$ and $B \subseteq \R^{\calV^\eps_R}$, we use the notations $\omega^p$ and $\Piv^p_x(B)$ from Subsubsection~\ref{sss.russo}. In particular, for every $x \in \mathcal{V}_R^\eps$ and every $p \in \R$ we write:
\[
\omega^p_x = \un_{ \{ X_x \geq -p \}} = \un_{\{f(x) \geq -p\}} \, .
\]
Also, we write $\cross^\eps(2R,R)$ for the subset of $\lbrace 0,1 \rbrace^{\mathcal{V}_R^\eps}$ that corresponds to the crossing from left to right of $[0,2R] \times [0,R]$. Note that $\cross^\eps_p(2R,R) = \lbrace \omega^p \in \cross^\eps(2R,R) \rbrace$. We have the following result (see Proposition~\ref{p.russo} for its link with the influences):
 
\begin{prop}\label{p.infl_decr}
Assume that $f$ satisfies Conditions~\ref{a.super-std},~\ref{a.std},~\ref{a.decay} as well as Condition~\ref{a.pol_decay} for some $\alpha > 4$. Then, there exist constants $\eta = \eta(\kappa) > 0$ and $C = C(\kappa) < +\infty$ such that, for every $\eps \in ]0,1]$, every $R \in ]0,+\infty[$, every $x \in \calV^\eps_R$ and every $p \in [-1,1]$ we have:
\[
\Pro \left[ \Piv_x^p(\cross^\eps(2R,R)) \cond f(x) = -p \right] \leq C \, R^{-\eta} \, .
\]
\end{prop}
Note that the constants in Proposition~\ref{p.infl_decr} do not depend on $p$ (as long $p$ belongs to a bounded interval, in our proof we have chosen $[-1,1]$).

\subsubsection{A bound on the infinite norm of the square root matrix}\label{sss.square_root}

In order to use Theorem~\ref{t.KMSnonproduct}, it is important to control the quantity $\|\sqrt{\Sigma}\|_{\infty,op}$. We are interested in the case $\Sigma=K_R^\eps:=$ the correlation matrix $K$ restricted to $\calV^\eps_R$ (see Subsubsection~\ref{sss.pol} the notation $\calV^\eps_R$). However, we were not able to estimate $\| \sqrt{K^\eps_R} \|_{\infty,op}$. Instead, we obtained estimates on $\| \sqrt{K^\eps} \|_{\infty,op}$ where $K^\eps$ is $K$ restricted to $\mathcal{V}^\eps$, by using Fourier techniques (which work only in a translation invariant setting, hence not for $K^\eps_R$). Here, $K^\eps$ is an infinite matrix, so let us be more precise about what we mean by square root of $K^\eps$: To define the product of two infinite matrices, we ask the infinite sums that arise to be absolutely convergent. A square root of $K^\eps$ is simply an infinite matrix $\sqrt{K^\eps}$ such that in this sense we have $\sqrt{K^\eps}^2 = K^\eps$. It may be that classical spectral theory arguments imply that such a square root exists. However, we will not need such arguments since we will have to construct quite explicitly $\sqrt{K^\eps}$ in order to estimate its infinite operator norm (where as in the finite dimensional case we let $\| \sqrt{K^\eps} \|_{\infty,op} := \underset{i}{\text{\textup{max}}} \sum_{j} |\sqrt{K^\eps}(i,j)|$, one difference being that this sum might be infinite). Let us also note that the matrices $\sqrt{K^\eps}$ that we will construct will be positive definite in the sense that their restriction to finite sets of indices is positive definite. By spectral theory arguments it may be that only one such matrix exists. However, we do not need such a result either.\\

Note that, since Theorem~\ref{t.KMSnonproduct} is stated for finite dimensional Gaussian fields, it would have been better to have an estimate on $\| \sqrt{K^\eps_R} \|_{\infty,op}$. See Lemma~\ref{l.dg_R/dp} for more details on this issue.\\

We will prove the following in Section \ref{s.sqrt}. The reason why our main result Theorem~\ref{t.main} is only proved for the Bargmann-Fock model is that we were unable to find a general estimate on $\|\sqrt{K^\eps}\|_{\infty,op}$ with explicit dependence on $\eps$ as $\eps\downarrow 0$ (this is the only place where we prove a result only for the Bargmann-Fock field). If one could prove a bound by $C \eps^{-2+\delta}$ for some $\delta > 0$ and $C<+\infty$ for another process (which also satisfies some of the conditions of Subsection~\ref{ss.cond}), the rest of the proof of Theorem~\ref{t.main} would piece together and yield the same result for this process. See Proposition~\ref{p.mainmoregnrl}.


\begin{prop}\label{p.square_root_estimate}
Assume that $f$ satisfies Condition~\ref{a.super-std} and Condition~\ref{a.fourier} for some $\alpha>5$. Then, for all $\eps>0$, there exists a symmetric square root $\sqrt{K^\eps}$ of $K^\eps$ such that:
\[
\| \sqrt{K^\eps} \|_{\infty,op} <+\infty\, .
\]
Assume further that $\kappa(x)=\exp \left( -\frac{1}{2}|x|^2 \right)$ (which is the covariance kernel function of the Bargmann-Fock field). Then, there exist $\eps_0>0$ and $C<+\infty$ such that, for all $\eps\in]0,\eps_0]$, there exists a symmetric square root $\sqrt{K^\eps}$ of $K^\eps$ such that:
\[
\| \sqrt{K^\eps} \|_{\infty,op} \leq C \, \frac{1}{\eps} \, \log\left(\frac{1}{\eps}\right) \, .
\]
\end{prop}

\subsubsection{Combining all the above results}\label{sss.combining}
The following Equation~\eqref{e.box-crossing_eps} is not used anywhere in the paper though it appears implicitely in Lemma~\ref{l.dg_R/dp}. However, since it sums up the crux of the proof, we state it here as a guide to the reader. Its proof could easily be extracted from the arguments below (see Remark~\ref{r.the_equation_for_constraint}).\\

As we will explain in Section~\ref{ss.mainpf}, by combining the results of Subsubsections~\ref{sss.russo},~\ref{sss.KMS},~\ref{sss.pol} and~\ref{sss.square_root}, we can obtain the following: Consider the Bargmann-Fock field and let $p \in ]0,1]$. There exists $C< +\infty$ such that, if $\eps \in ]0,1]$ and $R \in ]0,+\infty[$, then:
\begin{equation}\label{e.box-crossing_eps}
\Pro \left[ \cross_p^\eps(2R,R) \right] \geq 1 - C_1 \exp \left( -p \, c_2 \, \frac{\eps}{\log(\frac{1}{\eps})} \sqrt{ \log_+ \left( C_1 R^{-\eta} \frac{\log(\frac{1}{\eps})}{\eps} \right) } \right) \, ,
\end{equation}
where $C_1<+\infty$ and $c_2>0$ are some absolute constants. One can deduce from~\eqref{e.box-crossing_eps} that, if $\eps=\eps(R) \geq \log^{-(1/2-\delta)}(R)$, then $\Pro \left[ \cross_p^\eps(2R,R) \right]$ goes to $1$ as $R$ goes to $+\infty$ (sufficiently fast so that $\sum_k \Pro \left[ \cross_p^{\eps(2^k)}(2^{k+1},2^k) \right] < +\infty$ as required, see~\eqref{e.constraint1}).\\

We will also obtain results analogous to Equation~\eqref{e.box-crossing_eps} for more general fields and with $\|\sqrt{K^\eps}\|_{\infty,op}$ instead of $\frac{1}{\eps} \log(\frac{1}{\eps})$, which will enable us to prove~\eqref{e.criterion_eps} and thus obtain Theorem~\ref{t.epsKesten}.


\subsection{From discrete to continuous}\label{ss.from_discrete_to_cont}

It is helpful to keep in mind in this subsection that, as explained above, in the case of the Bargmann-Fock field, the quantity $\Pro \left[ \cross_p^\eps(2R,R) \right]$ goes to $1$ as $R\rightarrow +\infty$ if $\eps = \eps(R) \geq \log^{-(1/2-\delta)}(R)$. In order to obtain an analogous result for the continuum model, we need to measure the extent to which $\Pro \left[ \cross_p^\eps(2R,R) \right]$ approximates $\Pro \left[ \cross_p(2R,R) \right]$. To this purpose, it seems natural to use the approximation estimates based on the discretization schemes of~\cite{bg_16} which have been generalized in~\cite{bm_17}. More precisely, we think that the following can easily be extracted from~\cite{bm_17} (however, we do not write a formal proof since we will not need such a result):\footnote{Let us be a little more precise here: one can use for instance Propositions~6.1 and~6.3 of~\cite{bm_17} and (more classical) results about the regularity of $\calD_p$ as for instance Lemma~$\transversality$ of~\cite{rv_rsw}.} Assume that $f$ satisfies Condition~\ref{a.super-std} and that $\kappa$ is $C^6$. Then, for every $p \in \R$ and $\delta > 0$, there exists $C = C(\kappa,p,\delta) < +\infty$ such that for each $\eps\in]0,1]$ and each $R\in\eps\Z_{>0}$ we have:
\begin{equation}\label{e.discret_bm}
\Pro \left[\cross^\eps_p(2R,R) \bigtriangleup \cross_p(2R,R)\right]\leq C \eps^{2-\delta}R^2\, .
\end{equation}
This would imply in particular that, for every sequence $(\eps(R))_{R>0}$ such that: (i) for each $R>0$, $R \in \eps(R)\Z_{>0}$ and: (ii) $\eps(R) = O(R^{-(1+\delta)})$ for some $\delta > 0$, we have:
\[
\Pro \left[\cross^{\eps(R)}_p(2R,R) \bigtriangleup \cross_p(2R,R)\right] \underset{R \rightarrow +\infty}{\longrightarrow} 0 \, .
\]

Unfortunately, the constraint $\eps \leq R^{-(1+\delta)}$ is incompatible with $\eps \geq \log^{-(1/2-\delta)}(R)$. To solve this quandary, we will replace the discretization result with a sprinkling argument.\\

Let $p > 0$. The idea is to compare the probability of a continuum crossing at parameter $p$ to the probability of a discrete crossing at parameter $p/2$. Instead of using the discretization results of~\cite{bg_16,bm_17}, we will use the following result, which we prove in Subsection~\ref{s.sprinkling}. For references about the strategy of sprinkling in the context of Bernoulli percolation, see the notes on Section~2.6 of~\cite{grimmett1999percolation}.

\begin{prop}\label{p.sprinkling}
Assume that $f$ satisfies Condition~\ref{a.super-std} and that it is a.s. $C^2$. Let $p>0$. Then, there exist constants $C_1<+\infty$ and $c_2=c_2(\kappa,p)>0$ as well as $\eps_0=\eps_0(\kappa,p)>0$ such that for each $\eps \in ]0,\eps_0]$ and $R>1$:
\[\prob\left[\cross^\eps_{p/2}(2R,R) \setminus \cross_p(2R,R) \right]\leq C_1R^2\eps^{-2}\exp(-c_2\eps^{-4}) \, .\]
In particular, for every sequence $(\eps(R))_{R>0}$ such that $\eps(R) \geq \log^{-(1/4+\delta)}(R)$ for some $\delta > 0$, we have:
\[\prob\left[\cross^\eps_{p/2}(2R,R) \setminus \cross_p(2R,R) \right] \underset{R \rightarrow +\infty}{\longrightarrow} 0  \, .\]
\end{prop}

\begin{remark}
While this estimate may seem stronger than~\eqref{e.discret_bm}, we wish to emphasize that it only shows that discretization is legitimate in one direction and also involves a change in threshold from $-p/2$ to $-p$.
\end{remark}

If we choose for instance $\eps=\eps(R)=\log^{-1/3}(R)$, then:
\[
\log^{-(1/2-\delta)}(R) \ll \eps(R) \ll \log^{-(1/4+\delta)}(R)
\]
for any $\delta > 0$ small enough. Hence, with this choice of $\eps$ and for the Bargmann-Fock field, on the one hand for any $p>0$ we have a lower bound on $\Pro \left[ \cross^\eps_{p/2}(2R,R) \right]$ which is close to $1$ (see Subsubsection~\ref{sss.combining}), and on the other hand we can use Proposition~\ref{p.sprinkling} to go back to the continuum and obtain a lower bound on $\Pro \left[ \cross_p(2R,R) \right]$ which is close to $1$. More precisely, we will see in Section~\ref{ss.mainpf} that we can obtain:
\[
\Pro \left[ \cross_p(2R,R) \right] \geq 1- \exp \left( -c \, \log^{1/7}(R) \right) \, ,
\]
for some $c=c(p) > 0$, which implies that the box-crossing criterion of Lemma~\ref{l.criterion} is satisfied, and ends the proof of our main result Theorem~\ref{t.main}.

\section{Proofs of the phase transition theorems and of the exponential decay}\label{s.mainpf}

\subsection{Proof of the phase transition theorems}\label{ss.mainpf}

In this subsection, we combine the results stated in Section~\ref{s.strategy} in order to prove Theorem~\ref{t.main}. We also prove Theorem~\ref{t.epsKesten}. Most of the proof of the former carries over to that of  the latter. The first half of the proof is the following lemma, in which we work with the following function:

\begin{equation}\label{e.defi_g_R}
g_R^\eps : p \longmapsto \log \left(\frac{\Pro \left[ \cross_p^\eps\left(2R,R\right) \right]}{1-\Pro \left[ \cross_p^\eps\left(2R,R\right) \right]} \right) \, .
\end{equation}

\begin{lem}\label{l.dg_R/dp}
There exists an absolute constant $c>0$ such that the following holds: Assume that $f$ satisfies Condition~\ref{a.super-std}, and fix $\eps > 0$ and $R >0$. Moreover, let $\mu_R^\eps$ be the law of $f$ restricted to $\mathcal{V}_R^\eps$ (defined in Subsubsection~\ref{sss.pol}) and let $K^\eps$ be $K$ restricted to $\mathcal{V}^\eps$. Assume that there exists a symmetric square root $\sqrt{K^\eps}$ of $K^\eps$. Then:

\begin{equation}\label{e.formula_dg_R/dp}
\forall p \in \R, \, \frac{d}{dp}g_R^\eps(p) \geq c\, \|\sqrt{K^\eps}\|_{\infty,op}^{-1} \, \, \sqrt{\log_+ \left( \frac{1}{\| \sqrt{K^\eps} \|_{\infty,op} \cdot \underset{x \in \mathcal{V}^\eps_R}{\text{\textup{max}}} I_{x,\mu_R^\eps}\left(\cross_p^\eps(2R,R)\right)} \right)} \, .
\end{equation}
\end{lem}
Here, square root and infinite norm of an infinite matrix have the same meaning as in Subsubsection~\ref{sss.square_root}.\\

This lemma is mostly a consequence of Theorem~\ref{t.KMSnonproduct}. The only difficulty comes from the fact that, while Theorem~\ref{t.KMSnonproduct} deals with finite dimensional Gaussian fields, the correlation matrix on the right hand side of Equation~\eqref{e.formula_dg_R/dp} is the infinite matrix $K^\eps$ (which we need to apply Proposition~\ref{p.square_root_estimate} later in this section). To overcome such issues, we proceed by approximation and, instead of dealing with a $\mu_R^\eps$ variable directly, we apply Theorem~\ref{t.KMSnonproduct} to a Gaussian vector defined using $\sqrt{K^\eps}$ (namely, the Gaussian vector $Y^{\eps,\rho}$ below).

\begin{proof}[Proof of Lemma~\ref{l.dg_R/dp}]
Fix $\eps>0$ and $R > 0$. For each $\rho>0$, let $H^{\eps,\rho}$ be $\sqrt{K^\eps}$ restricted to $\calV^\eps \cap [-\rho,\rho]^2$, let $K^{\eps,\rho}=(H^{\eps,\rho})^2$, and let $Y^{\eps,\rho} \sim \calN(0,K^{\eps,\rho}) =: \mu^{\eps,\rho}$. A simple computation shows that since $K^\eps$ is non-degenerate,\footnote{Meaning that for any non-zero finitely supported $v\in \R^{\calV^\eps}$, $K^\eps v\neq 0$.} so is its square root. Restriction preserves non-degeneracy so $Y^{\eps,\rho}$ is indeed non-degenerate and we can apply Theorem~\ref{t.KMSnonproduct} to $\mu^{\eps,\rho}$ and the (increasing) event $\cross_p^\eps(2R,R)$. We obtain that, if $\rho$ is sufficiently large:
\begin{multline}\label{e.dg_R/dp.1}
\sum_{i\in\calV^\eps_R}I_{i,\mu^{\eps,\rho}}\left(\cross_p^\eps\left(2R,R\right)\right)\geq c \|H^{\eps,\rho}\|_{\infty,op}^{-1} \, \mu^{\eps,\rho}\left(\cross_p^\eps\left(2R,R\right)\right)\left(1-\mu^{\eps,\rho}\left(\cross_p^\eps\left(2R,R\right)\right)\right)\\
\times \sqrt{\log_+\left(\frac{1}{\|H^{\eps,\rho}\|_{\infty,op}\max_{i\in\calV^\eps_R}I_{i,\mu^{\eps,\rho}}(\cross_p^\eps\left(2R,R\right))}\right)}\, .
\end{multline}

Here we use the fact that, since $\cross_p^\eps\left(2R,R\right)$ depends only on the sites of $\calV^\eps_R$, the influences on this event of all the sites outside of this box vanish. Now, observe that:


\[
\frac{d}{dp}g_R^\eps(p) = \frac{d}{dp} \Pro \left[ \cross_p^\eps(2R,R) \right] \cdot \frac{1}{\Pro \left[ \cross_p^\eps(2R,R) \right] \, \left( 1-\Pro \left[ \cross_p^\eps(2R,R) \right] \right)} \, .
\]

By Proposition~\ref{p.russo},
\begin{equation}\label{e.dg_R/dp.2}
\frac{d}{dp}g_R^\eps(p) = \frac{1}{\Pro \left[ \cross_p^\eps(2R,R) \right] \, \left( 1-\Pro \left[ \cross_p^\eps(2R,R) \right] \right)}\sum_{i\in\calV^\eps_R}I_{i,\mu_R^\eps}\left(\cross_p^\eps\left(2R,R\right)\right)\, .
\end{equation}


Also, by definition of $\|\cdot\|_\infty$, for every $\rho$, we have:

\begin{equation}\label{e.dg_R/dp.3}
\|H^{\eps,\rho}\|_{\infty,op} \leq  \|\sqrt{K^\eps}\|_{\infty,op} \, .
\end{equation} 

In view of Equations~\eqref{e.dg_R/dp.1},~\eqref{e.dg_R/dp.2} and~\eqref{e.dg_R/dp.3}, we are done as long as we prove that:
\bi 
\item[1.] $\mu^{\eps,\rho} \left( \cross_p^\eps(2R,R) \right)  \underset{\rho \rightarrow +\infty}{\longrightarrow} \mu_R^\eps \left( \cross_p^\eps(2R,R) \right) ( = \Pro \left[ \cross^\eps_p(2R,R) \right])$ and:
\item[2.] $\forall i \in \calV_R^\eps, \, I_{i,\mu^{\eps,\rho}}\left(\cross_p^\eps\left(2R,R\right)\right) \underset{\rho \rightarrow +\infty}{\longrightarrow} I_{i,\mu_R^\eps}\left(\cross_p^\eps\left(2R,R\right)\right)\, .$
\ei
To this purpose, we need the following elementary lemma:
\begin{lemma}\label{l.cond.cont}
Let $(Y_\rho)_{\rho>0}$ be a sequence of non-degenerate Gaussian vectors in $\R^n$ with covariance $\Sigma_\rho$ and mean $m_\rho$, assume that $\Sigma_\rho$ converges to some invertible matrix $\Sigma$ as $\rho$ goes to $+\infty$, and that $m_\rho$ converges to some $m \in \R^n$ as $\rho$ goes to $+\infty$. Let $Y \sim \mathcal{N}(m,\Sigma) \,$. Then, for any Borel subset $A\subset\R^n$, any $i\in\{1,\cdots,n\}$ and any $q\in\R$, we have:
\begin{align*}
\prob[Y_\rho\in A]&\underset{\rho \rightarrow +\infty}{\longrightarrow}\prob[Y\in A] \, ,\\
\prob \left[Y_\rho\in A \cond Y_\rho(i)=q\right]&\underset{\rho \rightarrow +\infty}{\longrightarrow}\prob\left[Y\in A \cond Y(i)=q\right]\, .
\end{align*}
\end{lemma}
\begin{proof}
For the first statement, just notice that the density function of a non-degenerate Gaussian vector is a continuous function of its covariance and of its mean, and apply dominated convergence. By Proposition~1.2 of \cite{azws}, the law of $Y_\rho$ conditioned on the value of $Y_\rho(i)$ is that of a Gaussian vector whose mean and covariance depend continuously on $\Sigma_\rho$ and $m_\rho$. Note also that the law $(Y_\rho(j))_{j \neq i}$ when we condition on $\lbrace Y_\rho(i) = q \rbrace$ (respectively the law $(Y(j))_{j \neq i}$ when we condition on $\lbrace Y(i) = q \rbrace$) is still non-degenerate.\footnote{To see this, complete the vector $e_i$ into an orthogonal basis for $\Sigma_\rho$ (respectively $\Sigma$) and express $Y_\rho$ (respectively $Y$) in this basis.} We conclude by applying the first statement.
\end{proof}
Item~1 above is a direct consequence of the first part of Lemma~\ref{l.cond.cont}. Regarding Item~2, note that, since $\cross_p^\eps\left(2R,R\right)$ is an increasing threshold event, by Proposition~\ref{p.russo}, we have the following equalities:
\begin{multline*}
I_{i,\mu^{\eps,\rho}}\left(\cross_p^\eps\left(2R,R\right)\right)\\
=\Pro \left[ Y^{\eps,\rho} \in \piv_i\left(\cross_p^\eps\left(2R,R\right)\right) \cond Y^{\eps,\rho}(i)=-p\right]\frac{1}{\sqrt{2\pi K^{\eps,\rho}(i,i)}} e^{-p^2/(2\, K^{\eps,\rho}(i,i))}\, ,
\end{multline*}
and:
\begin{multline*}
I_{i,\mu_R^\eps}\left(\cross_p^\eps\left(2R,R\right)\right)\\
=\Pro \left[ Y_R^\eps \in \piv_i\left(\cross_p^\eps\left(2R,R\right)\right) \cond Y_R^\eps(i)=-p\right]\frac{1}{\sqrt{2\pi K_{R}^\eps(i,i)}}e^{-p^2/(2\, K_{R}^\eps(i,i))}\, ,
\end{multline*}
where $Y^\eps_R \sim \calN(0,K^\eps_R)$. We conclude by applying the second part of Lemma~\ref{l.cond.cont} (to $Y^{\eps,\rho}$ restricted to $\mathcal{V}_R^\eps$ and to $K^\eps_R$).
\end{proof}

We are now ready to prove our main result.

\begin{proof}[Proof of Theorem \ref{t.main}]
We consider the Bargmann-Fock field $f$. By Theorem~\ref{t.c_RSW}, for every $p\leq 0$, a.s. there is no unbounded connected component in $\calD_p$. We henceforth consider some parameter $p_0 \in ]0,1]$. By Lemma~\ref{l.criterion}, it is enough to prove that $f$ satisfies criterion~\eqref{e.criterion} (note that if we prove that this criterion is satisfied for every $p_0 \in ]0,1]$ then we obtain the result for every $p_0$ since the quantities $\Pro \left[ \neg \cross_{p_0}(2R,R) \right]$ are non-increasing in $p_0$). To this end we first fix $R_0<+\infty$ to be determined later, we let $R>R_0$, and we use the discretization procedure introduced in Subsection~\ref{ss.discretization} with:
\[
\eps=\eps(R)=\log^{-1/3}(R) \, .
\]
Let $g_R:=g_R^{\eps(R)}$ be as in~\eqref{e.defi_g_R}. We are going to apply Lemma~\ref{l.dg_R/dp}. Let us estimate the quantities that appear in this lemma. By Proposition~\ref{p.square_root_estimate}, there exists a constant $C<+\infty$ independent of $R$ such that:
      \[
      \|\sqrt{K^{\eps(R)}}\|_{\infty,op}\leq C \frac{1}{\eps(R)}\log\left(\frac 1 {\eps(R)}\right) .
      \]
      Moreover by Propositions~\ref{p.russo} and~\ref{p.infl_decr}, if $R_0$ is large enough, there exists $\eta>0$ independent of $R > R_0$ and of $p \in [-1,1]$ such that:
       \begin{equation}\label{e.infl_in_proof}
      \forall x\in\mathcal{V}^{\eps(R)}_R,\ I_{x,\mu_R^{\eps(R)}}\left(\cross_p^{\eps(R)}(2R,R)\right)\leq R^{-\eta} \, .
      \end{equation}
     In particular, $\| \sqrt{K^{\eps(R)}} \|_{\infty,op} \cdot \underset{x \in \mathcal{V}^{\eps(R)}_R}{\text{\textup{max}}} I_{x,\mu_R^{\eps(R)}}\left(\cross_p^{\eps(R)}(2R,R)\right)<1$ if $R_0$ is large enough, and by applying Lemma~\ref{l.dg_R/dp} we obtain:
      \begin{align*}
        \frac{d}{dp}g_R(p)&\geq \frac{c}{C} \eps(R) \log^{-1}\left(\frac 1 {\eps(R)}\right)\sqrt{\eta\log(R)-\log\left(\frac{1}{C\eps(R)}\right)-\log\left(\log\left(\frac{1}{\eps(R)}\right)\right)}\\
        &\geq \frac{c}{3C}\log^{-1/3}(R)\log^{-1}\left(\log(R)\right)\sqrt{\eta\log(R)-\frac{1}{3}\log\left(\frac{\log(R)}{C}\right)-\log(3\log(\log(R)))}\\
        &\geq \frac{c}{3C}\sqrt{\eta/2}\log^{1/6}(R)\log^{-1}\left(\log(R)\right)\\
        &\geq \frac{c}{C}\sqrt{\eta/2}\log^{1/7}(R) \, .
      \end{align*}
      By integrating from $0$ to $p_0/2$, we get:
      \begin{align*}\log\left(\frac{1}{1-\Pro\left[\cross_{p_0/2}^{\eps(R)}(2R,R)\right]}\right)&\geq\log\left(\frac{\Pro\left[\cross_{p_0/2}^{\eps(R)}(2R,R)\right]}{1-\Pro\left[\cross_{p_0/2}^{\eps(R)}(2R,R)\right]}\right)\\
        &\geq \frac{p_0}{2} \frac{c}{C}\sqrt{\eta/2}\log^{1/7}(R)+\log
        \left(\frac{\Pro\left[\cross_0^{\eps(R)}(2R,R)\right]}{1-\Pro\left[\cross_0^{\eps(R)}(2R,R)\right]}\right) \, .
      \end{align*}
     By Theorem~\ref{t.epsRSW}, if $R_0$ is large enough, there exists $C'<+\infty$ independent of $R>R_0$ such that:
      \[\log\left(\frac{\Pro\left[\cross_0^{\eps(R)}(2R,R)\right]}{1-\Pro\left[\cross_0^{\eps(R)}(2R,R)\right]}\right)\geq -C'\,.\]
      Therefore:
      \begin{equation}\label{e.mainpf.subexp}
        \Pro\left[\cross_{p_0/2}^{\eps(R)}(2R,R)\right]\geq1-\exp\left(-\frac{c}{2C}\sqrt{\eta/2} \, p_0\log^{1/7}(R)+C'\right)\,.
      \end{equation}
      Now:
      \[\Pro\left[\cross_{p_0}(2R,R)\right]\geq\Pro\left[\cross_{p_0/2}^{\eps(R)}(2R,R)\right]-\Pro\left[\cross_{p_0/2}^{\eps(R)}(2R,R)\setminus[\cross_{p_0}(2R,R)\right] \, ,\]
and by Proposition~\ref{p.sprinkling}, if $R_0$ is large enough, there exist constants $c_1>0$ and $C_2=C_2(p_0)<+\infty$ indepedent of $R>R_0$ such that:
        \begin{align*}
          \Pro\left[\cross_{p_0/2}^{\eps(R)}(2R,R)\setminus[\cross_{p_0}(2R,R)\right]&\leq C_2R^2\eps(R)^{-2}\exp\left(-c_1\eps(R)^{-4}\right)\\
            &\leq C_2\exp\left(2\log(R)+(2/3)\log(\log(R))-c_1\log^{4/3}(R)\right)\\
            &\leq C_2\exp\left(-\frac{c_1}{2}\log^{4/3}(R)\right).
        \end{align*}
        Combining this estimate with Equation~\eqref{e.mainpf.subexp} we get:
        \[\Pro\left[\cross_{p_0}(2R,R)\right]\geq 1-\exp\left(-\frac{c}{2C}\sqrt{\eta/2} \, p_0\log^{1/7}(R)+C'\right)-C_2\exp\left(-\frac{c_1}{2}\log^{4/3}(R)\right).\,\]
        All in all, for a large enough choice of $R_0$, there exists $c_3=c_3(p_0)>0$ such that for each $R > R_0$,
        \[\Pro\left[\cross_{p_0}(2R,R)\right]\geq 1-\exp\left(-c_3\log^{1/7}(R)\right)\,.\]
        Hence, criterion~\ref{e.criterion} is satisfied and we are done.
        \end{proof}
        
\begin{remark}
Note that this proof also implies that:
\begin{equation}\label{e.cross.decay}
\forall p>0,\, \lim_{R\rightarrow+\infty}\cross_p(2R,R)=1\, ,
\end{equation}
which will be useful for the proof of Theorem \ref{t.exp}.
\end{remark}

\begin{remark}\label{r.the_equation_for_constraint}
Note that if we follow the above proof without taking $\eps=\log^{-1/3}(R)$  but by taking any $\eps \in ]0,1]$ and $R$ larger than some constant that does not depend on $\eps$, we obtain~\eqref{e.box-crossing_eps}.
\end{remark}

The only result that we have used in the proof of Theorem~\ref{t.main} and that we have managed to prove only for the Bargmann-Fock field is the $\frac{1}{\eps}\log\left(\frac 1 \eps\right)$ estimate on the infinite operator norm of $\sqrt{K^\eps}$. We have the following more general result:

\begin{prop}\label{p.mainmoregnrl}
Assume that $f$ satisfies Conditions~\ref{a.super-std},~\ref{a.std},~\ref{a.decay} as well as Condition~\ref{a.pol_decay} for some $\alpha > 4$. Assume also that there exist $\delta > 0$ and $C < +\infty$ such that for every $\eps$ sufficiently small $K^\eps$ admits a symmetric square root and:
\begin{equation}\label{e.2-delta}
\|\sqrt{K^\eps}\|_{\infty,op}\leq C \frac{1}{\eps^{2-\delta}} \, .
\end{equation}
Then the critical threshold for the continuous percolation model induced by $f$ is $0$. More precisely, the probability that ${\cal D}_p$ has an unbounded connected component is $1$ if $p>0$, and $0$ if $p\leq 0$. Moreover, in the case where $p>0$, the component is a.s. unique.
\end{prop}

\begin{remark}
A way to prove that~\eqref{e.2-delta} is to use our results of Section~\ref{s.sqrt} (where we assume that $f$ satisifies Condition~\ref{a.super-std} and Condition~\ref{a.fourier} for some $\alpha > 5$): Define the quantity $\Upsilon(\eps)$ as in Lemma~\ref{l.sqrt.2}. If one proves that there exists $\delta > 0$ and $C_1<+\infty$ such that, for every $\eps$ sufficiently small, we have $\Upsilon(\eps) \leq C_1 \eps^{-1+\delta}$, then one obtains that there exists $C_2<+\infty$ such that, for every $\eps$ sufficiently small:
\[
\|\sqrt{K^\eps}\|_{\infty,op}\leq C_2 \frac{1}{\eps^{2-\delta}} \log \left( \frac{1}{\eps} \right) \, .
\]
\end{remark}

\begin{proof}[Proof of Proposition~\ref{p.mainmoregnrl}]
The proof is almost identical to that of Theorem~\ref{t.main} so we will not detail elementary computations. First fix some $h>0$ such that:
\begin{equation}\label{e.cond_on_h}
(-1/4-h)(2-\delta)+1/2 > 0 \, .
\end{equation}
By Theorem~\ref{t.c_RSW}, for every $p\leq 0$, a.s. there is no unbounded connected component in $\calD_p$. We henceforth consider some parameter $p_0 \in ]0,1]$. Fix $R_0<+\infty$ to be determined later, and let:
\[
\eps=\eps(R)=\log^{-1/4-h}(R)\, .
\]
By Propositions~\ref{p.russo} and~\ref{p.infl_decr} if $R_0$ is large enough, there exists $\eta>0$ independent of $R> R_0$ and of $p \in [-1,1]$ such that:
\[
\forall x\in\mathcal{V}^{\eps(R)}_R,\ I_{x,\mu_R^{\eps(R)}}\left(\cross_p^{\eps(R)}(2R,R)\right)\leq R^{-\eta} \, .
\]
Hence, by Lemma~\ref{l.dg_R/dp}, if $R_0$ is large enough:
\[
\frac{d}{dp}g_R(p) \geq  \frac{c}{C} \log^{(-1/4-h)(2-\delta)+1/2}(R) \sqrt{\eta/2} \log^{1/2}(R) \, .
\]
By Theorem~\ref{t.epsRSW}, if $R_0$ is large enough, there exists $C'<+\infty$ independent of $R>R_0$ such that:
      \[
      \log\left(\frac{\Pro\left[\cross_0^{\eps(R)}(2R,R)\right]}{1-\Pro\left[\cross_0^{\eps(R)}(2R,R)\right]}\right)\geq -C'\,.
      \]
Integrating from $0$ to $p_0/2$ we get:
 \[
        \Pro\left[\cross_{p_0/2}^{\eps(R)}(2R,R)\right]\geq1-\exp\left(- \frac{c}{2C} \, \frac{p_0}{2} \log^{(-1/4-h)(2-\delta)+1/2}(R) \sqrt{\eta/2}+C'\right)\,.
      \]
Now, if we apply Proposition~\ref{p.sprinkling}, we obtain that, if $R_0$ is large enough, there exist constants $c_1=c_1(\kappa)>0$ and $C_2=C_2(\kappa,p_0)<+\infty$ independent of $R>R_0$ such that:
\begin{align*}
& \Pro \left[ \cross_{p_0}(2R,R) \right]\\
& \geq \Pro \left[ \cross_{p_0/2}^{\eps(R)}(2R,R) \right] - C_2 \exp \left( -c_1 \log^{1+4h}(R) \right)\\
& \geq 1-\exp\left(- \frac{c}{2C} \, p_0 \log^{(-1/4-h)(2-\delta)+1/2}(R) \sqrt{\eta/2}+C'\right) - C_2 \exp \left( -c_1 \log^{1+4h}(R) \right) \, .
\end{align*}
Since (by~\eqref{e.cond_on_h}) we have $(-1/4-h)(2-\delta)+1/2 > 0$, the criterion of Lemma~\ref{l.criterion} is satisfied and we obtain that a.s. ${\cal D}_{p_0}$ has a unique unbounded component.
\end{proof}

The following proof is also almost identical to that of Theorem~\ref{t.main} except that we do not have to use discretization estimates to conclude.

\begin{proof}[Proof of Theorem \ref{t.epsKesten}]
Fix $\eps \in ]0,1]$. By Theorem~\ref{t.cor_bg_eps}, for every $p\leq 0$, a.s., there is no unbounded black component. From now on, we fix some $p_0 \in ]0,1]$. By Lemma~\ref{l.criterion_eps}, it is enough to prove that $f$ satisfies criterion \eqref{e.criterion_eps}. To this end we first fix $R_0>0$ to be determined later and we let $R>R_0$. By Proposition \ref{p.square_root_estimate}, there exists $C=C(\kappa,\eps)<+\infty$ such that:
      \[
            \|\sqrt{K^\eps}\|_{\infty,op}\leq C \, .
            \]
Moreover by Propositions \ref{p.russo} and \ref{p.infl_decr} if $R_0$ is large enough, there exists $\eta>0$ independent of $R> R_0$ and $p \in [-1,1]$ such that:
\[\forall x\in\mathcal{V}^\eps_R,\ I_{x,\mu_R^\eps}\left(\cross_p^\eps(2R,R)\right)\leq R^{-\eta} \, .\]
Hence, by Lemma~\ref{l.dg_R/dp}, if $R_0$ is large enough:
\[
\frac{d}{dp}g_R(p)\geq \frac{c}{C}\sqrt{\frac{\eta}{2}\log(R)} \, .
\]
 By Theorem~\ref{t.epsRSW}, if $R_0$ is large enough, there exists $C'<+\infty$ independent of $R>R_0$ such that:
      \[\log\left(\frac{\Pro\left[\cross_0^\eps(2R,R)\right]}{1-\Pro\left[\cross_0^\eps(2R,R)\right]}\right)\geq -C'\,.\]
Integrating from $0$ to $p_0$ we get:
 \[
        \Pro\left[\cross_{p_0}^\eps(2R,R)\right]\geq1-\exp\left(-\frac{c}{C} p_0 \sqrt{\frac{\eta}{2} \log(R)}+C'\right)\,.
      \]
      Hence, $f$ satisfies criterion~\ref{e.criterion_eps} and we are done.
\end{proof}

\subsection{Proof of the exponential decay}\label{ss.exp}

In this subsection, we prove Theorem~\ref{t.exp} by following classical arguments of planar percolation that involve a recursion on crossing probabilities. In order to start the induction, we use the following property: According to Equation~\eqref{e.cross.decay}, if $f$ is the Bargmann-Fock field, then:

\begin{equation}\label{e.start}
\forall p>0,\, \Pro [\cross_p(2R,R)]\underset{R\rightarrow+\infty}{\longrightarrow}1\, .
\end{equation}

\begin{proof}[Proof of Theorem~\ref{t.exp}]
Let $f$ be the Bargmann-Fock field. Let us prove that for every $p>0$, there exists $c=c(p)>0$ such that for every $R > 0$ we have:

\begin{equation}\label{e.mainpf.1}
\Pro \left[ \cross_p(2 R,R) \right] \geq 1-\exp(-c \, R) \, .
\end{equation}

Theorem~\ref{t.exp} will follow from \eqref{e.mainpf.1}. Indeed, by a simple gluing argument, for every $\rho_1 > 1$ and $\rho_2 > 0$, there exists $N = N(\rho_1,\rho_2) \in \Z_{>0}$ such that, for every $R > 0$, there exist $N$ $\rho_1 R \times R$ rectangles such that if these rectangles are crossed lengthwise then $[0,\rho_2 R] \times [0,R]$ is also crossed lengthwise.\\

In order to prove~\eqref{e.mainpf.1}, we follow classical ideas in planar percolation theory, see for instance the second version of the proof of Theorem~10 in~\cite{bollobas2006percolation} or the proof of Proposition~3.1 in~\cite{ahlberg2016sharpness}. To begin with, define the sequence $(r_k)_{k\geq 0}$ as follows. $r_0=1$ and for each $k\in\N$,
\[
r_{k+1}=2r_k+\sqrt{r_k}\, .
\]
From this definition it easily follows that for each $k$, $r_k\geq 2^k$, which in turn implies that $r_{k+1}\leq(2+2^{-k/2})r_k$. Using this last relation one can easily show that $r_k=O(2^k)$ so that there exists $C<+\infty$ such that for each $k\in\N$,
\begin{equation}\label{e.mainpf.4}
2^k\leq r_k\leq C2^k\, .
\end{equation}
Write:
\[
a_k = 1-\Pro \left[ \cross_p(2r_k,r_k) \right] + \exp(- \frac{r_k}{10}) \, .
\]
We will show that there exists $k_0=k_0(p)<+\infty$ such that for any $k\geq k_0$:
\begin{equation}\label{e.mainpf.2}
a_{k+1}\leq 49 a_k^2\, .
\end{equation}
By Equation~\eqref{e.start}, $\lim_{k\rightarrow+\infty}a_k=0$. This observation combined with Equation~\eqref{e.mainpf.2} yields~\eqref{e.mainpf.1} for $R=r_k$ with $k\geq k_0$ by an elementary induction argument. But since by \eqref{e.mainpf.4}, the sequence $(r_k)_{k\geq 0}$, grows geometrically, one then obtains \eqref{e.mainpf.1} for any $R>0$ by elementary gluing arguments.\\

In order to prove Equation~\eqref{e.mainpf.2} we first introduce two events, see Figure~\ref{f.cross}. First, the event $\fivecross_p(k)$ is the event that:
\begin{itemize}
\item the $2r_k \times r_k$ rectangles $[(ir_k,(i+2)r_k] \times [0,r_k]$ for $i=0,1,2,3$ are crossed from left to right by a continuous path in $\mathcal{D}_p$;
\item the $r_k \times r_k$ squares $[jr_k,(j+1)r_k] \times [0,r_k]$ for $j=1,2,3$ are crossed from top to bottom by a continuous path in $\mathcal{D}_p$.
\end{itemize}
Second, the event $\fivecross'_p(k)$ is the event $\fivecross_p(k)$ translated by $(0,r_k+\sqrt{r_k})$. Note that $\fivecross_p(k) \cup \fivecross'_p(k) \subseteq \cross_p(5r_k,2r_k+\sqrt{r_k})$ Moreover, there exists $k_1<+\infty$ such that for each $k\geq k_1$, $5r_k\geq 2r_{k+1}=4r_{k+1}+2\sqrt{r_k}$ so that $\cross_p(5r_k,2r_k+\sqrt{r_k})\subseteq \cross_p(2r_{k+1},r_{k+1})$. Hence, for each $k\geq k_1$:
\begin{equation}\label{e.mainpf.3}
\Pro \left[ \cross_p(2r_{k+1},r_{k+1}) \right] \geq 1- \Pro \left[ \neg \fivecross_p(k) \cap \neg \fivecross_p'(k) \right] \, .
\end{equation}
\begin{figure}[!h]
\includegraphics[scale=1]{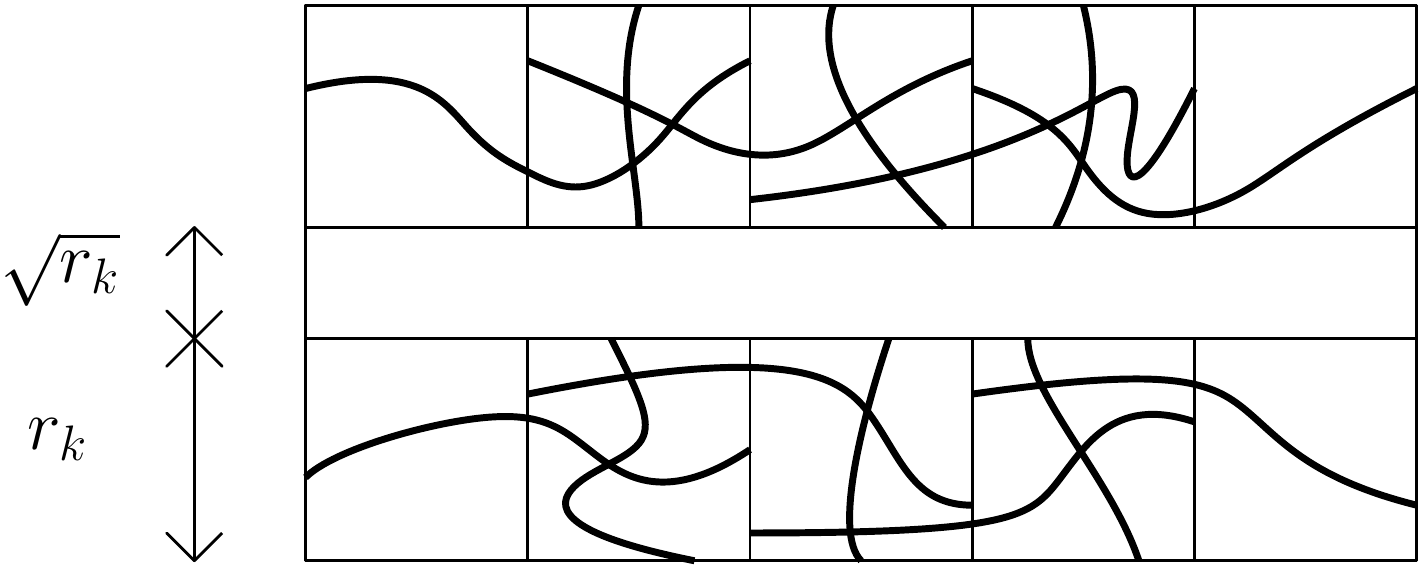}
\caption{The events $\fivecross_p(k)$ and $\fivecross_p'(k)$.}\label{f.cross}
\end{figure}

We claim that the events $\fivecross_p(k)$ and $\fivecross_p'(k)$ are asymptotically independent. More precisely, the following is a direct consequence of Theorem~$\quasiindependence$ from~\cite{rv_rsw}.
\begin{claim}\label{cl.quasi_ind}
There exists $k_2 < +\infty$ such that, for every $p \in \R$ and every $k\geq k_2$:
\begin{multline*}
\left| \Pro \left[ \neg \fivecross_p(k) \cap \neg \fivecross_p'(k) \right] - \Pro \left[ \neg \fivecross_p(k) \right] \, \Pro \left[ \neg \fivecross_p'(k) \right] \right|\\
\leq O(r_k^4) e^{-\frac{\sqrt{r_k}^2}{2}} \leq e^{-\frac{r_k}{4}} \, .
\end{multline*}
\end{claim}
Let $k \geq \max\{k_1,k_2\}$. Applying Claim~\ref{cl.quasi_ind} in Equation~\eqref{e.mainpf.3}, we get:
\begin{eqnarray*}
\Pro \left[ \cross_p(2r_{k+1},r_{k+1}) \right] & \geq & 1- \Pro \left[ \neg \fivecross_p(k) \right] \cdot \Pro \left[ \neg \fivecross_p'(k) \right] - e^{-\frac{r_k}{4}}\\
& = & 1-\Pro \left[ \neg \fivecross_p(k) \right]^2 -e^{-\frac{r_k}{4}} \text{ (by stationarity).}
\end{eqnarray*}
By a union bound we have:
\[
\Pro \left[ \neg \fivecross_p(k) \right] \leq 4(1-\Pro \left[ \cross_p(2r_k,r_k) \right]) + 3 (1-\Pro \left[ \cross_p(r_k,r_k) \right]) \leq 7(1-\Pro \left[ \cross_p(2r_k,r_k) \right]) \, .
\]
Thus:
\[
1-\Pro \left[ \cross_p(2r_{k+1},r_{k+1}) \right] \leq \big( 7(1-\Pro \left[ \cross_p(2r_k,r_k) \right]) \big)^2 + e^{-\frac{r_k}{4}} \, ,
\]
and:
\begin{eqnarray*}
49(a_k)^2 & \geq & \big( 7(1-\Pro \left[ \cross_p(2r_k,r_k) \right]) \big)^2 + 49e^{-\frac{2r_k}{10}}\\
& \geq & 1-\Pro \left[ \cross_p(2r_{k+1},r_{k+1}) \right] + 49e^{-\frac{2r_k}{10}} - e^{-\frac{r_k}{4}}\\
& \geq & 1-\Pro \left[ \cross_p(2r_{k+1},r_{k+1}) \right] + e^{-\frac{2r_k+\sqrt{r_k}}{10}} \text{ if } k \text{ is sufficiently large}\\
& = & a_{k+1} \, .
\end{eqnarray*}
This is exactly Equation~\eqref{e.mainpf.2}.
\end{proof}

\section{Percolation arguments}\label{s.perco}

In this section we prove Lemma~\ref{l.criterion} and Proposition~\ref{p.infl_decr}. In doing so we also obtain Lemma~\ref{l.criterion_eps}.\\

For every $x\in\R^2$ and every $\rho>0$, we set $B(x,\rho)=x+[-\rho,\rho]^2$.
For every $x\in\R^2$ $0<\rho_1\leq \rho_2$, we set:
\[
\ann(\rho_1,\rho_2)=[-\rho_2,\rho_2]^2\setminus]\rho_1,\rho_1[^2 \, ,
\]
and $\ann(x,\rho_1,\rho_2)=x+\ann(\rho_1,\rho_2)$.
\subsection{Proof of Lemma~\ref{l.criterion}}\label{ss.criterion}

In this subsection, we prove Lemma~\ref{l.criterion}. The proof of Lemma~\ref{l.criterion} also works for Lemma~\ref{l.criterion_eps} with only a few obvious changes. For the proofs of these results, we never use the fact that our Gaussian field is non-degenerate.

\begin{proof}[Proof of Lemma~\ref{l.criterion}]
In this proof, we write $\cross_p(k) = \cross_p(2^{k+1},2^k)$ and we write $\cross_p'(k)$ for the event that there is a continuous path joining the bottom side of the rectangle $[0,2^k] \times [0,2^{k+1}]$ to its top side in $\calD_p\cap [0,2^k] \times [0,2^{k+1}]$. By translation invariance and $\frac{\pi}{2}$-rotation invariance, $\Pro \left[ \cross_p'(k) \right] = \Pro \left[ \cross_p(k) \right]$ for every $k$. Thus:
\[
\sum_{k \in \N} \Pro \left[ \neg \left( \cross_p(k) \cap \cross_p'(k) \right) \right] \leq 2 \sum_{k \in \N} \Pro \left[ \neg \cross_p(k) \right] < +\infty\text{ by assumption.}
\]
Together with Borel-Cantelli lemma, it implies that a.s. there exists $k_0 \in \N$ such that, for every $k \geq k_0$, $\cross_p(k) \cap \cross_p'(k)$ is satisfied. Note that any crossing from left to right of $[0,2^{k+1}] \times [0,2^k]$ and any crossing from top to bottom of $[0,2^{k+1}] \times [0,2^{k+2}]$ must intersect. Similarly, any crossing from top to bottom of $[0,2^k] \times [0,2^{k+1}]$ and any crossing from left to right of $[0,2^{k+2}] \times [0,2^{k+1}]$ must intersect. All these intersecting crossings then form an unbounded connected component in $\mathcal{D}_p$.\\

Let us end the proof by showing that this unbounded component is a.s. unique. The proof follows the same overall structure: First, write $\Circ_p(k)$ for the event that there is a circuit in $\mathcal{D}_p \cap \ann(2^k,2^{k+1})$ surrounding the square $[-2^k,2^k]^2$. Note that, if eight well-chosen $2^{k+1} \times 2^k$ rectangles are crossed lengthwise, two well-chosen $2^k \times 2^k$ squares are crossed from left to right, and two well-chosen $2^k \times 2^k$ squares are crossed from top to bottom, then $\Circ_p(k)$ holds (see Figure~\ref{f.circ} where we already used such a property). Since the probability that a $2^k \times 2^k$ square is crossed is at least equal to the probability that a $2^{k+1} \times 2^k$ rectangle is crossed lengthwise, we get:
\[
\Pro \left[ \neg \Circ_p(k) \right] \leq 12 \Pro \left[ \neg \cross_p(k) \right] \, .
\]
Thus, as before,
\[
\sum_{k\in\N}\Pro \left[ \neg \Circ_p(k) \right]<+\infty \, ,
\]
and Borel-Cantelli lemma implies that a.s. there exists $k_0$ such that, for every $k \geq k_0$, $\Circ_p(k)$ holds. Now, note that, for any unbounded connected component, there exists $k_1$ such that this component crosses the annuli $\ann(2^k,2^{k+1})$ for every $k \geq k_1$. Thus, this component contains any circuit of this annulus for every $k \geq k_0 \vee k_1$. In particular, it must be unique.
\end{proof}

\begin{figure}
\begin{center}
\includegraphics[scale=0.7]{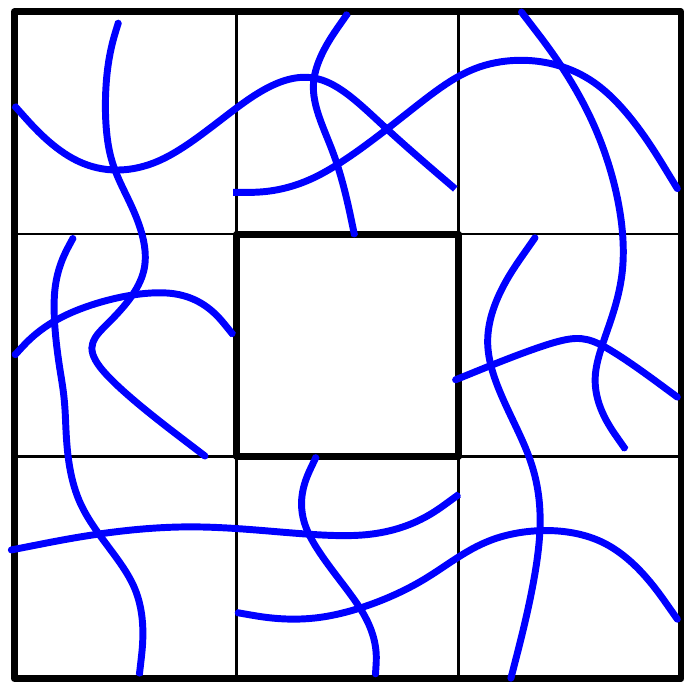}
\end{center}
\caption{If eight $2^{k+1} \times 2^k$ well chosen rectangles and four well chosen $2^k \times 2^k$ squares are crossed, then there is a circuit in $\ann(2^k,2^{k+1})$.}\label{f.circ}
\end{figure}

\subsection{Proof of Proposition \ref{p.infl_decr}}\label{ss.infl_decr}

The goal of this subsection is to prove Proposition~\ref{p.infl_decr}. To this purpose, we use the following result from~\cite{rv_rsw} which is a consequence of the discrete RSW estimate and quasi-independence. Such a result goes back to~\cite{bg_16} for $\alpha > 16$. Let $\eps > 0$ and consider the discrete percolation model defined in Subsection~\ref{ss.discretization}. Let $\arm^\eps_p(R)$ (respectively $\arm^{\eps,*}_p(R)$) be the event that in this model there is a black (respectively white) path in the annulus $\ann(1,R)$ from the inner boundary of this annulus to its outer boundary. We have the following:
\begin{prop}[\cite{bg_16} for $\alpha > 16$,~\cite{rv_rsw}]\footnote{More precisely, this is Proposition~$\polynomially$ of~\cite{rv_rsw}. Moreover, this can be extracted from the proof of Theorem~5.7 of~\cite{bg_16} for $\alpha>16$ and with slightly different assumptions on the differentiability and the non-degeneracy of $\kappa$.}\label{p.arm}
Assume that $f$ satisfies Conditions~\ref{a.super-std},~\ref{a.std},~\ref{a.decay} as well as Condition~\ref{a.pol_decay} for some $\alpha > 4$.  Then, there exists $C=C(\kappa)<+\infty$ and $\eta= \eta(\kappa) > 0$ such that, for each $\eps \in ]0,1]$, for each $R \in ]0,+\infty[$:
\[
\Pro \left[ \arm_0^\eps(R) \right], \, \Pro \left[ \arm_0^{*,\eps}(R) \right] \leq C \, R^{-\eta} \, .
\]
\end{prop}
To deduce Proposition~\ref{p.infl_decr} from Proposition~\ref{p.arm}, we need the following lemma and more particularly its consequence Corollary~\ref{c.mon.conditional}.
\begin{lem}\label{l.mon.conditional}
Let $\Sigma = \left( \Sigma(i,j) \right)_{0 \leq i,j \leq n}$ be a $(n+1) \times (n+1)$ symmetric positive definite matrix, let $(m_0,m) = (m_0,m_1, \cdots, m_n) \in \R^{n+1}$ and let $(X_0,X) = (X_0,X_1, \cdots, X_n) \sim \mathcal{N} \left( (m_0,m),\Sigma \right)$. Assume that $\Sigma(0,i) \geq 0$ for every $i \in \lbrace 1, \cdots, n \rbrace$. Then, for every non-decreasing\footnote{For the partial order $(x_1, \cdots, x_n) \leq (y_1, \cdots, y_n)$ if $x_i \leq y_i$ for every $i \in \lbrace 1, \cdots, n \rbrace$.} function $\varphi : \R^n \rightarrow \R$, the following quantity is non-decreasing in $q$:
\[
\E \left[ \varphi(X) \cond X_0 = q \right] \, .
\]
Similarly, if $\varphi$ is non-increasing, then this quantity is non-increasing in $q$.
\end{lem}
\begin{proof}
This is a direct consequence of the following result (see for instance Proposition~1.2 of~\cite{azws}): Under the probability measure $\Pro \left[ \cdot \cond X_0 = q \right]$, $X$ is a Gaussian vector whose covariance matrix does not depend on $q$ and whose mean is:
\[
m + (q-m_0) v \, ,
\]
where $v = \left( \frac{\Sigma(0,i)}{\Sigma(0,0)} \right)_{1 \leq i \leq n}$ (which has only non-negative entries).
\end{proof}

From this lemma we deduce the following:
\begin{cor}\label{c.mon.conditional}
Let $(X_0,X)$ be as in Lemma~\ref{l.mon.conditional} and let $\varphi : \R^n \rightarrow \R$ be a non-decreasing function. Then, for every $p \in \R$ we have:
\[
\E \left[ \varphi(X) \cond X_0 = -p \right] \leq \E \left[ \varphi(X) \cond X_0 \geq -p \right] \, .
\]
Similarly, if $\varphi$ is non-increasing, then:
\[
\E \left[ \varphi(X) \cond X_0 = -p \right] \leq \E \left[ \varphi(X) \cond X_0 \leq -p \right] \, .
\]
\end{cor}
\begin{proof}
Assume that $\varphi$ is non-decreasing and let $\gamma_0$ be the density of $X_0$, which exists because $\Sigma$ is positive definite. We have :
\begin{align*}
\E \left[ \varphi(X) \cond X_0 \geq -p \right] & =\frac{1}{\Pro \left[ X_0 \geq -p \right]}\int_{-p}^{+\infty}\E\left[\varphi(X)\cond X_0=t\right]\gamma_0(t)dt\\
& \geq \frac{1}{\Pro \left[ X_0 \geq -p \right]}\E\left[\varphi(X)\cond X_0=-p\right]\int_{-p}^{+\infty}\gamma_0(t)dt\\
& = \E\left[\varphi(X)\cond X_0=-p\right]
\end{align*}
where the inequality comes from Lemma~\ref{l.mon.conditional} applied with $q=-p$.
\end{proof}
We are now ready to prove Proposition \ref{p.infl_decr}. As usual in percolation theory, we are going to control our  $p$-dependent probabilities by probabilities at the self-dual point $p=0$.
\begin{proof}[Proof of Proposition~\ref{p.infl_decr}] Fix $R\geq 2$ and $\eps \in ]0,1]$. Let us prove the proposition in the case $p \geq 0$. Let $x \in \mathcal{V}^\eps_R$ and let $\arm^{\eps,*}_p(x,R/2)$ be the event that there is a white path in $\ann\left(x,1,R/2\right)$ from $\partial B(x,1)$ to $\partial B(x,R/2)$. We claim that $\Piv_x^p(\cross^\eps(2R,R)) \subseteq \arm^{\eps,*}_p(x,R/2)$. Indeed, $\Piv_x^p(\cross^\eps(2R,R))$ is the event that there are two white paths from the top and bottom sides of the rectangle $[0,2R] \times [0,R]$ to two neighbors of $x$ and two black paths from the left and right sides of $[0,2R] \times [0,R]$ to two other neighbors of $x$ (with an exception when $x$ does not belong to $[0,2R] \times [0,R]$ but is a neighbor of a point $y \in [0,2R] \times [0,R]$ such that the edge between $x$ and $y$ crosses the left or right side of $[0,2R] \times [0,R]$; in this case $\Piv_x^p(\cross^\eps(R))$ is the event that there is a white path from top to bottom that goes through $y$ and a black path from $y$ to the oppposite side). In particular, at least one black path reaches a point at distance at least $R/2$ of $x$.\\

Since $\arm^{\eps,*}_p(x,R/2)$ is decreasing and does not depend on $f(x)$, then (by Corollary~\ref{c.mon.conditional} and with $a:=\Pro \left[ Z \leq -1 \right]$ where $Z \sim \calN(0,1)$):
\begin{eqnarray*}
\Pro \left[ \arm^{\eps,*}_p(x,R/2) \cond f(x) = -p \right] & \leq & \Pro \left[ \arm^{\eps,*}_p(x,R/2) \cond f(x) \leq -p \right]\\
& \leq & \frac{\Pro \left[ \arm^{\eps,*}_p(x,R/2) \right]}{\Pro \left[ f(x) \leq -p \right]}\\
& \leq &  a^{-1} \Pro \left[ \arm^{\eps,*}_p(x,R/2) \right] \text{ since } p \leq 1\\
& \leq & a^{-1} \Pro \left[ \arm^{\eps,*}_0(x,R/2) \right] \text{since $p\geq 0$}\, .
\end{eqnarray*}
The result follows from Proposition~\ref{p.arm} (and stationarity). The case $p \leq 0$ is treated similarly (by noting that $\Piv_x^p(\cross^\eps(2R,R)) \subseteq \arm^{\eps}_p(x,R/2)$ and by studying this increasing event exactly as above).
\end{proof}

\section{A KKL theorem for biased Gaussian vectors: the proof of Theorem \ref{t.KMSnonproduct}}\label{s.KMS}
In this section we prove Theorem~\ref{t.KMSnonproduct}. The proofs presented here do not rely on any results of the other sections. Recall that we use the following definition for influence. Given a vector $v\in\R^n$ and a Borel probability measure $\mu$ on $\R^n$ the influence of $v$ on $A$ under $\mu$ is  
\[
I_{v,\mu}(A)=\underset{r \downarrow 0}{\liminf}\frac{\mu \left( A+[-r,r]v \right) - \mu \left( A \right)}{r} \in [0,+\infty]\, .
\]
\subsection{Sub-linearity of influences implies Theorem \ref{t.KMSnonproduct}}\label{ss.KMSnonproduct}

The aim of this subsection is to prove Theorem~\ref{t.KMSnonproduct}. That this theorem holds for product Gaussian measures is a result by Keller, Mossel and Sen,~\cite{keller2012geometric} (see Corollary~\ref{c.KMS} below). In order to extend this result to general Gaussian measures, we use a sub-linearity property for influences. In the present subsection, we state this property and use it to derive Theorem~\ref{t.KMSnonproduct} from the product case. In Subsection~\ref{ss.monotonic}, we prove the sub-linearity property for monotonic sets.

\paragraph{A KKL theorem for product Gaussian vectors} The authors of \cite{keller2012geometric} introduce and study the notion of \textit{geometric influences} which are defined as follows: Let $\nu$ be a probability measure on $\R$ and let $A$ be a Borel subset of $\R^n$. If $i \in \lbrace 1, \cdots, n \rbrace$, let:
\[
A_i^x := \lbrace y \in \R \, : \, (x_1, \cdots, x_{i-1}, y, x_{i+1}, \cdots, x_n) \in A \rbrace \, .
\]
The geometric influence of $i$ on $A$ under the measure $\nu^{\otimes n}$ is:
\[
I^\mathcal{G}_{i,\nu}(A) := \E_{x \sim \nu^{\otimes n}} \left[ \nu^+(A_i^x) \right] \in [0,+\infty] \, ,
\]
where $\nu^+$ is the lower Minkowski content, defined as follows: for all $B \subseteq \R$ Borel,
\[
\nu^+(B) := \underset{r \downarrow 0}{\liminf} \frac{\nu\left(B+[-r,r]\right) - \nu \left(B\right)}{r} \in [0,+\infty] \, .
\]

In the case where $\mu=\nu^{\otimes n}$, $I^\calG_{i,\nu}$ and $I_{i,\mu}$ are closely related. Indeed, firstly, by Fubini's Theorem and Fatou's lemma, for each Borel subset $A\subset\R^n$ and each $i\in\{1,\dots,n\}$,

\begin{equation}\label{e.KMS.Fatou}
I_{i,\nu^{\otimes n}}(A)=\liminf_{r\downarrow 0} \E_{x \sim \nu^{\otimes n}} \left[ \frac{\nu(A_i^x+[-r,r]) - \nu(A)}{r} \right]\geq I^\calG_{i,\nu}(A) \, .
\end{equation}

While the reverse inequality seems not true in general, we expect it to hold for a wide class of events. In particular, our Lemma~\ref{l.conv} and~(2.6) from~\cite{keller2012geometric} imply that this is the case for monotonic events. Since it is not useful to us, we do not investigate this matter any further.\\


We will need the following result, which is a direct consequence of Item~(1) of Theorem~1.5 of \cite{keller2012geometric}. Several results of this type can also be found in the more recent~\cite{cordero2012hypercontractive} (see for instance the paragraph above Corollary~7 therein).

\begin{thm}\label{t.KMS}
There exists an absolute constant $c > 0$ such that the following holds:

Let $\nu = \mathcal{N}(0,1)$ and let $A$ be a Borel measurable subset of $\R^n$. Then:
\[
\sum_{i=1}^n I^\mathcal{G}_{i,\nu}(A) \geq c \, \nu^{\otimes n}(A) \, (1-\nu^{\otimes n}(A)) \, \sqrt{\log_+ \left( \frac{1}{\underset{i \in \lbrace 1, \cdots, n \rbrace}{\text{\textup{max}}} I^\mathcal{G}_{i,\nu}(A)}\right) } \, .
\]
\end{thm}

\begin{cor}\label{c.KMS}
Let $\nu$ and $A$ be as in Theorem~\ref{t.KMS}. Then:
\[
\sum_{i=1}^n I_{i,\nu^{\otimes n}}(A) \geq c \, \nu^{\otimes n}(A) \, (1-\nu^{\otimes n}(A)) \, \sqrt{\log_+ \left( \frac{1}{\underset{i \in \lbrace 1, \cdots, n \rbrace}{\text{\textup{max}}} I_{i,\nu^{\otimes n}}(A)}\right) } \, .
\]
\end{cor}
\begin{proof}
This is a direct consequence of Theorem~\ref{t.KMS} and Equation~\eqref{e.KMS.Fatou}. 
\end{proof}

Now, let us state a sub-linearity property for the influences that we will be proved in Subsection~\ref{ss.monotonic}.

\begin{prop}\label{p.sublin}
Let $\Sigma$ be a $n \times n$ symmetric positive definite matrix, let $\mu \sim \mathcal{N}(0,\Sigma)$ and let $\left( e_1, \cdots, e_n \right)$ be the canonical basis of $\R^n$. Also, let $v = \sum_{j=1}^n v_i \, e_i \in \R^n$. For every $i \in \lbrace 1, \cdots, n \rbrace$ and every $A$ monotonic Borel subset of $\R^n$, we have:
\begin{equation}\label{e.sublinrec}
I_{v,\mu}(A) \leq \sum_{i=1}^n |v_i| \, I_{i,\mu}(A) \, .
\end{equation}
\end{prop}
We expect the above result to hold for a larger class of Borel subsets $A$ (in particular, we have not found any examples of Borel sets for  which this does not hold) and not only for the decomposition on the canonical basis. See Appendix A of Chapter 7 of \cite{alelele} or Appendix A of Chapter 2 of \cite{hugogogo} for a proof of the inequality $I_{u+v,\mu}(A)\leq I_{u,\mu}(A)+I_{v,\mu}(A)$ for any $v,w$ when $A$ is a semi-algebraic set.\\

Let us now derive Theorem~\ref{t.KMSnonproduct} from Corollary \ref{c.KMS} using Proposition\ref{p.sublin}.

\begin{proof}[Proof of Theorem~\ref{t.KMSnonproduct}] Let $\sqrt{\Sigma}$ be a symmetric square root of $\Sigma$ and let $\nu=\mathcal{N}(0,1)$. Then, $\mu=\calN(0,\Sigma)$ is the pushforward measure of $\nu^{\otimes n}$ by $\sqrt{\Sigma}$. Thanks to Corollary~\ref{c.KMS} (applied to the event $\sqrt{\Sigma}^{-1}(A)$), it is sufficient to prove the following claim.
\end{proof}
\begin{claim}
Let $\nu = \mathcal{N}(0,1)$. We have:
\begin{equation}\label{e.cons_sublin_1}
\max_{j\in\lbrace 1,\cdots, n\rbrace} I_{j,\nu^{\otimes n}}(\sqrt{\Sigma}^{-1} \left( A \right)) \leq || \sqrt{\Sigma} ||_{\infty,op} \cdot \underset{i \in \lbrace 1, \cdots, n \rbrace}{\max} I_{i,\mu}(A) \, .
\end{equation}
Moreover:
\begin{equation}\label{e.cons_sublin_2}
\sum_{j=1}^n I_{j,\nu^{\otimes n}}(\sqrt{\Sigma}^{-1} \left( A \right)) \leq || \sqrt{\Sigma} ||_{\infty,op} \cdot \sum_{i=1}^n I_{i,\mu}(A) \, .
\end{equation}
\end{claim}
\begin{proof}
For each $j\in\lbrace 1,\cdots, n\rbrace$:
\begin{eqnarray*}
I_{j,\nu^{\otimes n}}(\sqrt{\Sigma}^{-1} \left( A \right)) & = & \underset{r \downarrow 0}{\liminf} \frac{\nu^{\otimes n} \left( \sqrt{\Sigma}^{-1}  \left( A \right) + [-r,r] \, e_j \right) - \nu^{\otimes n} \left( \sqrt{\Sigma}^{-1} \left( A \right) \right)}{r}\\
& = & \underset{r \downarrow 0}{\liminf} \frac{\mu \left( A + [-r,r] \, \sqrt{\Sigma} \cdot e_j \right) - \mu \left( A \right)}{r}\\
& = & I_{\sqrt{\Sigma} \cdot e_j,\mu}(A) \, .
\end{eqnarray*}
By using Proposition~\ref{p.sublin}, we obtain that:
\begin{equation}\label{e.cons_sublin_inter}
I_{j,\nu^{\otimes n}}(\sqrt{\Sigma}^{-1} \left( A \right)) \leq \sum_{i = 1}^n |\sqrt{\Sigma}(i,j)| I_{i,\mu}(A) \, .
\end{equation}
Hence (since $\sqrt{\Sigma}$ is symmetric):
\[
I_{j,\nu^{\otimes n}}(\sqrt{\Sigma}^{-1} \left( A \right)) \leq \sum_{i = 1}^n |\sqrt{\Sigma}(i,j)| \underset{i \in \lbrace 1, \cdots, n \rbrace}{\max} I_{i,\mu}(A) \leq ||\sqrt{\Sigma} ||_{\infty,op} \, \underset{i \in \lbrace 1, \cdots, n \rbrace}{\max} I_{i,\mu}(A) \, .
\]
We obtain ~\eqref{e.cons_sublin_1} by taking the supremum over $j$. Inequality~\eqref{e.cons_sublin_inter} also implies that:
\begin{eqnarray*}
\sum_{j=1}^n I_{\sqrt{\Sigma} \cdot e_j,\mu}(A) & \leq & \sum_{j=1}^n \sum_{i=1}^n |\sqrt{\Sigma}(i,j)| \, I_{i,\mu}(A)\\
& = & \sum_{i=1}^n I_{i,\mu}(A) \sum_{j=1}^n |\sqrt{\Sigma}(i,j)|\\
& \leq & ||\sqrt{\Sigma}||_{\infty,op} \cdot \sum_{i=1}^n I_{i,\mu}(A) \, ,
\end{eqnarray*}
which is~\eqref{e.cons_sublin_2}.
\end{proof}

\subsection{Sub-linearity of influences for monotonic events}\label{ss.monotonic}

The proof of Proposition~\ref{p.sublin} relies on the following lemma.

\begin{lem}\label{l.conv}
Let $\Sigma$ be a $n \times n$ symmetric positive definite matrix and let $\mu = \mathcal{N}(0,\Sigma)$. Moreover, let $v \in \R^n$. For every $A$ monotonic Borel subset of $\R^n$ and every $i \in \lbrace 1, \cdots, n \rbrace$, we have:
\[
\exists \, \underset{r \downarrow 0}{\text{\textup{lim}}\;} \frac{\mu \left( A + [-r,r] e_i + [-r,r]v  \right) - \mu \left( A +  [-r,r]v \right)}{r} = I_{i,\mu}(A) \, .
\]
\end{lem}

We first prove Proposition \ref{p.sublin} using Lemma \ref{l.conv}.

\begin{proof}[Proof of Proposition~\ref{p.sublin}]
Let $A$ be a monotonic Borel subset of $\R^n$ and fix $v \in \R^n$. Let $v^i=\sum_{k=i}^n v_k e_k$. We will prove that, for every $i \in \{ 1, \cdots, n-1 \}$:
\begin{equation}\label{e.sublin}
I_{v^i,\mu}(A) \leq |v_i| \, I_{i,\mu}(A) + I_{v^{i+1},\mu}(A) \, .
\end{equation}
The result will then follow directly by induction (and since $I_{v^{n},\mu}(A)= I_{n,\mu}(A)$). For any $w_1,w_2 \subseteq \R^n$, we have $[-r,r](w_1+w_2) \subseteq [-r,r]w_1+[-r,r]w_2$. Hence:

\begin{align*}
\mu \left( A + [-r,r]v^i \right) =\ &\mu \left( A + [-r,r] \left(v_i e_i + v^{i+1} \right) \right)\\
\leq\, &\mu \left( A + [-|v_i| r,|v_i| r] e_i + [-r,r] v^{i+1} \right)\\
=\ &\mu \left( A + [-|v_i| r,|v_i| r] e_i + [-r,r] v^{i+1} \right)\\
&\ - \mu \left( A + [-r,r] v^{i+1} \right) + \mu \left( A + [-r,r] v^{i+1} \right) \, .
\end{align*}

By Lemma \ref{l.conv} we have:
\[
\frac{\mu \left( A + [-|v_i| r,|v_i| r] \, e_i + [-r,r] \, v^{i+1} \right) - \mu \left( A + [-r,r]v^{i+1} \right)}{r} \underset{r \downarrow 0}{\longrightarrow} |v_i| \, I_{i,\mu}(A) \, .
\]
Equation~\eqref{e.sublin} follows.
\end{proof}

\begin{proof}[Proof of Lemma \ref{l.conv}] We are inspired by the proof of Proposition~1.3 in~\cite{keller2012geometric}. We write the proof for $A$ decreasing since the proof for $A$ increasing is identical. Also, we prove the result in the case $i = n$. We write $\tilde{x}$ for the first $(n-1)$ coordinates of any $x\in\R^n$. Let $A_r = A + [-r,r]v$, note that $A_r$ is decreasing, and write for any $\tilde{x}\in\R^{n-1}$:
\begin{align*}
s(\tilde{x}) &:= \sup \lbrace x_n \in \R \, : \, (\tilde{x}, x_n) \in A \rbrace \in [-\infty,+\infty] \, ,\\
s_r(\tilde{x}) &:= \sup \lbrace x_n \in \R \, : \, (\tilde{x}, x_n) \in A_r \rbrace \in [-\infty,+\infty] \, .
\end{align*}
Let $\lambda$ be the density function of $\mu$. Since $A$ and $A_r$ are decreasing, we have:
\begin{align}\label{e.mon.integral}
\frac{\mu ( A + [-r,r] e_n ) - \mu (A)}{r} = \frac{1}{r} \int_{\R^{n-1}}  \left(\int_{s(\tilde{x})}^{s(\tilde{x})+r} \lambda(\tilde{x},x_n) dx_n\right)d\tilde{x} \, , \nonumber\\
\frac{\mu ( A_r + [-r,r] e_n ) - \mu ( A_r )}{r} = \frac{1}{r} \int_{\R^{n-1}}  \left(\int_{s_r(\tilde{x})}^{s_r(\tilde{x})+r} \lambda(\tilde{x},x_n) dx_n\right)d\tilde{x} \, ,
\end{align}
where by convention $\int_{-\infty}^{-\infty+r} = \int_{+\infty}^{+\infty+r} = 0$. For each $\tilde{x}\in\R^{n-1}$ let:
\[
g_1(\tilde{x})=\sup_{x_n\in\R}\lambda(\tilde{x},x_n);\: \: g_2(\tilde{x})=\sup_{x_n\in\R}\left|\frac{\partial\lambda}{\partial x_n}(\tilde{x},x_n)\right|.
\]
Direct computation shows that $g_1,g_2\in L^1(\R^{n-1})$. By the mean value inequality, for each $\tilde{x}\in\R^{n-1}$:
\[
\left|\frac{1}{r}\left(\int_{s(\tilde{x})}^{s(\tilde{x})+r} \lambda(\tilde{x},x_n) dx_n\right)-\lambda(\tilde{x},s(\tilde{x}))\right|+\left|\frac{1}{r}\left(\int_{s_r(\tilde{x})}^{s_r(\tilde{x})+r} \lambda(\tilde{x},x_n) dx_n\right)-\lambda(\tilde{x},s_r(\tilde{x}))\right|
\]
is no greater than $2rg_2(\tilde{x})$. Combining this with Equation~\eqref{e.mon.integral} we get:
\begin{multline*}
\left|\frac{\mu ( A + [-r,r] e_n ) - \mu (A)}{r}- \frac{\mu ( A_r + [-r,r] e_n ) - \mu ( A_r )}{r} \right|\\
\leq \left|\int_{\R^{n-1}}\lambda(\tilde{x},s(\tilde{x}))
-\lambda(\tilde{x},s_r(\tilde{x})) \, d\tilde{x}\right|+2r\int_{\R^{n-1}}g_2(\tilde{x}) \, d\tilde{x}\, .
\end{multline*}
Since $g_2\in L^1(\R^{n-1})$, the second integral in the last inequality is finite and independent of $r$. Moreover:
\[
\int_{\R^{n-1}}\lambda(\tilde{x},s_r(\tilde{x}))
-\lambda(\tilde{x},s(\tilde{x})) \, d\tilde{x}=\int_{\R^{n-1}}\int_{s(\tilde{x})}^{s_r(\tilde{x})}\frac{\partial\lambda}{\partial x_n}(\tilde{x},x_n) \, dx_nd\tilde{x}\, .
\]
Since $\frac{\partial\lambda}{\partial x_n}\in L^1(\R^n)$, by dominated convergence, all that remains is to show that for a.e. $\tilde{x}\in\R^{n-1}$: $s_r(\tilde{x}) \underset{r\downarrow 0}{\longrightarrow} s(\tilde{x})$. Since for each $\tilde{x}$ the sequence $s_r(\tilde{x})$ is decreasing, it converges to some $s_\infty(\tilde{x}) \geq s(\tilde{x})$. Let us prove that, for a.e. $\tilde{x}\in\R^{n-1}$, $s_\infty(\tilde{x})=s(\tilde{x})$. To do so, first note that, since $A$ is decreasing, we have:
\[
0 \leq \mu(A_r) - \mu(A) \leq \Pro \left[ X - \sum_{i=1}^n r |v_i|e_i \in A \right] - \Pro \left[ X \in A \right] \, ,
\]
where $X \sim \calN(0,\Sigma)$. By dominated convergence, the right hand side tends to $0$ when $r\rightarrow 0$. Now, note that:
\[
\mu(A_r)-\mu(A) = \int_{\R^{n-1}} \int_{s(\tilde{x})}^{s_r(\tilde{x})} \lambda(\tilde{x},x_n) \, dx_n d\tilde{x} \, .
\]
Hence, by Fatou's lemma:
\[
0 = \int_{\R^{n-1}} \int_{s(\tilde{x})}^{s_\infty(\tilde{x})} \lambda(\tilde{x},x_n) \, dx_n d\tilde{x} \, .
\]
Since the $\lambda$ takes only positive values, this implies that for a.e. $\tilde{x}\in\R^{n-1}$, $s_\infty(\tilde{x})=s(\tilde{x})$.
\end{proof}

\section{An estimate on the infinite operator norm of square root of infinite matrices: the proof of Proposition \ref{p.square_root_estimate}}\label{s.sqrt}

The goal of this section is to prove Proposition~\ref{p.square_root_estimate}. We assume that $f$ satisfies Condition~\ref{a.super-std} and Condition~\ref{a.fourier} for some $\alpha>5$. In particular, the Fourier transform of $\kappa$ takes only positive values. Moreover, we assume that $\kappa$ is $C^3$ and there exists $C<+\infty$ such that for every $\beta \in \N^2$ such that $\beta_1 + \beta_2 \leq 3$, we have:
\begin{equation}\label{e.decay_derivative}
|\partial^\beta \kappa(x)| \leq C |x|^{-\alpha} \, ,
\end{equation}
for some $\alpha > 5$. In this subsection, we never use the fact that our Gaussian field is non-degenerate.\\

Recall that for each $\eps>0$, $\calT^\eps$ is the lattice $\calT$ scaled by a factor $\eps$, that $\calV^\eps$ is the set of vertices of $\calT^\eps$, and that $K^\eps$ is the restriction of $K$ to $\calV^\eps$. We begin by observing that the face-centered square lattice $\calT$, when rotated by $\frac{\pi}{4}$ and rescaled by $\sqrt{2}$, has the same vertices as $\Z^2$. Since by Condition~\ref{a.fourier} our field is invariant by $\frac{\pi}{4}$-rotation, \textbf{we will} simply \textbf{replace $\mathcal{T}$ by $\Z^2$ throughout the rest of this section}.

Let $\T^2$ be the flat 2D torus corresponding to the circle of length $2\pi$. Throughout this section, we will identify $\lambda \Z^2$-periodic functions on $\R^2$ (for some $\lambda>0$) with functions on $\lambda \T^2$ and their integrals over the box $[-\lambda\pi,\lambda\pi]^2$ with integrals over $\lambda\T^2$. We will use the following convention for the Fourier transform:
\[
\forall\xi\in\R^2,\ \hat{\kappa}(\xi)=\frac{1}{4\pi^2}\int_{\R^2} e^{-i\langle\xi,x\rangle}\kappa(x)dx \, .
\]

Let us begin with a sketch of the construction of the square root $\sqrt{K^\eps}$:
\bi 
\item[1.] First note that if we find some symmetric function $\eta_\eps \, : \, \Z^2 \rightarrow \R^2$ such that:
\[
\forall m \in \Z^2, \, \eta_\eps * \eta_\eps(m) := \sum_{m' \in \Z^2} \eta_\eps(m') \eta_\eps(m-m') = \kappa(\eps m) \, ,
\]
then $(\eps m_1, \eps m_2) \in \eps \Z^2 \longmapsto \eta_\eps(m_1-m_2)$ is a symmetric square root of $K^\eps$.
\item[2.] Let $\kappa_\eps$ be $\kappa$ restricted to $\eps \Z^2$ and let us try to construct $\eta_\eps$ above. The first idea is that, if $\eta_\eps * \eta_\eps = \kappa_\eps$, then the Fourier transform of $\eta_\eps$ should the square root of the Fourier transform of $\kappa_\eps$. In other words:
\[
\sfF(\eta_\eps) = \sqrt{\sfF(\kappa_\eps)} \, ,
\]
where $\sfF(\eta_\eps)(\xi) = \sum_{m \in \Z^2} \eta_\eps(m) e^{-i<\xi,m>}$ and similarly for $\sfF(\kappa_\eps)$.
\item[3.] Thus:
\begin{eqnarray*}
\eta_\eps(m) & = & \sfF^{-1} \left( \sqrt{\sfF(\kappa_\eps)} \right)\\
& = & \frac{1}{4\pi^2}\int_{\T^2} e^{i\langle \xi,m \rangle} \sqrt{\sfF(\kappa_\eps)}(\xi) \, d\xi \, .
\end{eqnarray*}
\item[4.] In the expression above, it seems difficult to deal with the term $\sfF(\kappa_\eps)$. To simplify the expression, we can use the Poisson summation formula and deduce that:
\[
\sfF(\kappa_\eps)(\xi) = 4\pi^2\eps^{-2}\sum_{m \in \Z^2} \hat{\kappa}(\eps^{-1}(2\pi m-\xi)) \, .
\]
For the Bargmann-Fock process, $\hat{\kappa}$ is well known since $\kappa$ is simply the Gaussian function.
\ei

There will be two main steps in the proof of Proposition~\ref{p.square_root_estimate}. First, we will make the above construction of $\eta_\eps$ rigorous by considering the four items above in the reverse order. More precisely, we will first set $\lambda_\eps$ as in Lemma~\ref{l.sqrt.0} below, then we will apply the Poisson summation formula to prove that $2\pi \eps^{-1} \lambda_\eps=\sqrt{\sfF(\kappa_\eps)}$. Next, we will define $\eta_\eps$ as in Lemma~\ref{l.sqrt.0}, we will show that $\sfF(\eta_\eps)=\sqrt{\sfF(\kappa_\eps)}$, and we will conclude that $\eta_\eps$ is a convolution square root of $\kappa_\eps$. All of this will be done in the proof Lemma~\ref{l.sqrt.0}. Secondly, we will prove estimates on $\sum_{m \in \Z^2} |\eta_\eps(m)|$ when $\kappa(x) = e^{-\frac{1}{2}|x|^2}$. This will be the purpose of Lemma~\ref{l.sqrt} (and most of the proof of this lemma will be written in a more general setting than $\kappa(x) = e^{-\frac{1}{2}|x|^2}$).\\

In the proofs, we will use a subscript $\eps$ to denote functions $: \, \Z^2 \rightarrow \R$ or to denote functions defined on $\T^2$ (or equivalently functions $2\pi \Z^2$-periodic). On the other hand, we will use a superscript $\eps$ to denote functions $: \, \eps^{-1}\Z^2 \rightarrow \R$ or to denote functions defined on $\eps^{-1}\T^2$ (or equivalently functions $2\pi\eps^{-1} \Z^2$-periodic).\\

Proposition~\ref{p.square_root_estimate} is a direct consequence of the following Lemmas~\ref{l.sqrt.0} and~\ref{l.sqrt}.

\begin{lem}\label{l.sqrt.0}
Assume that $f$ satisfies Condition~\ref{a.super-std} as well as Condition~\ref{a.std} for $\alpha > 5$. Fix $\eps > 0$ and let:
\[
\lambda_\eps \, : \, \xi \in \R^2 \mapsto \sqrt{\sum_{m\in\Z^2}\hat{\kappa}(\eps^{-1}(\xi-2\pi m))} \, .
\]
Then, $\lambda_\eps$ is a $C^3$, positive, even and $2\pi \Z^2$-periodic function. Next, define:
\[
\eta_\eps \, : \, m \in \Z^2 \mapsto \frac{1}{\eps(2\pi)}\int_{\T^2}e^{i\langle m,\xi\rangle}\lambda_\eps(\xi) \, d\xi \, .
\]
Then, $(\eps m_1, \eps m_2) \in \eps \Z^2 \longmapsto \eta_\eps(m_1-m_2)$ is a symmetric square root of $K^\eps$ and we have:
\[
\sum_{m \in \Z^2} |\eta_\eps(m)| < +\infty \, .
\]
\end{lem}

\begin{lemma}\label{l.sqrt}
Assume that $\kappa(x)=e^{-\frac{1}{2}|x|^2}$. Then, there exist constants $C_0<+\infty$ and $\eps_0>0$ such that for $\eps \in ]0,\eps_0]$:
\[
  \sum_{m\in\Z^2}|\eta_\eps(m)|\leq C_0\frac{1}{\eps} \log \left(\frac{1}{\eps} \right) \, .
  \]
\end{lemma}

\begin{proof}[Proof of Lemma \ref{l.sqrt.0}]
First note that~\eqref{e.decay_derivative} implies that $\hat{\kappa}$ is $C^3$ and that for every $\beta \in \N^2$ such that $\beta_1+\beta_2 \leq 3$, we have:
\begin{equation}\label{e.F(k)}
|\partial^\beta \hat{\kappa}(\xi)| \leq C' |\xi|^{-3} \, ,
\end{equation}
for some $C'=C'(\kappa)<+\infty$. This implies that the series under the square root of the definition of $\lambda_\eps$ converges in $C^3$-norm towards a $C^3$ function. Moreover, this series is clearly $2\pi \Z^2$-periodic, even, and positive (since $\hat{\kappa}$ takes only positive values). Thus, $\lambda_\eps$ is well defined and also satisfies these properties.\\

Let us now prove the second part of the lemma. For each $\eps>0$ let $\kappa_\eps:\Z^2\rightarrow\R$ be defined as $\kappa_\eps(m)=\kappa(\eps m)$. The discrete Fourier transform of $\kappa_\eps$ is:
\[
\sfF(\kappa_{\eps}) : \xi \in \T^2 \longmapsto \sum_{m\in\Z^2}\kappa_\eps(m)e^{-i\langle\xi,m\rangle}
\]
(since $\kappa$ and its derivatives of order up to $3$ decay polynomially fast with an exponent larger than $5$, the above series converges in $C^3$-norm). Now, since $\kappa$ and $\hat{\kappa}$ decay polynomially fast with an exponent larger than $2$, we can apply the Poisson summation formula (see\footnote{Be aware that the conventions used in \cite{grafakos} are different from ours.} Theorem~3.1.17 of \cite{grafakos}) which implies that:
\[
\forall \xi \in \T^2, \,  \sfF(\kappa_\eps)(\xi)=4\pi^2\eps^{-2}\sum_{m\in\Z^2}\hat{\kappa}(\eps^{-1}(2\pi m-\xi)) = 4\pi^2\eps^{-2} \lambda_\eps(\xi)^2 \, .
\]

As a result:

\[
\eta_\eps(m)=\frac{1}{4\pi^2}\int_{\T^2}\sqrt{\sfF(\kappa_\eps)}(\xi) e^{i\langle m,\xi\rangle}d\xi \, .
\]
In other words, the  $\eta_\eps(m)$'s are the Fourier coefficients of the $C^3$, positive and $2\pi \Z^2$-periodic function $\sqrt{\sfF(\kappa_\eps)}$, which implies in particular that $|\eta_\eps(m)| \leq C'' |m|^{-3}$ for some $C''=C''(\kappa,\eps)<+\infty$. As a result:
\begin{equation}\label{e.finite_sum}
\sum_{m \in \Z^2}|\eta_\eps(m)|<+\infty \, .
\end{equation}
Thanks to~\eqref{e.finite_sum}, we can apply the Fourier inversion formula (see for instance Proposition~3.1.14 of~\cite{grafakos}) which implies that:
\begin{equation}\label{e.sfeta_et_sfkappa}
\sfF(\eta_\eps) = \sqrt{\sfF(\kappa_\eps)} \, .
\end{equation}

Now, let us use the convolution formula (see for instance  Paragraph~1.3.3 of~\cite{rud_fou}). Since $\sum_{m \in \Z^2}|\eta_\eps(m)|<+\infty$, we have:
\begin{equation}\label{e.convol_converges}
\sum_{m \in \Z^2} \sum_{m' \in \Z^2} |\eta_\eps(m')\eta_\eps(m-m')| <+\infty \, ,
\end{equation}
and:
\begin{equation}\label{e.convol}
\sfF ( \eta_\eps * \eta_\eps ) = \sfF(\eta_\eps)^2 \, ,
\end{equation}
where $\eta_\eps * \eta_\eps \, : \, m \in \Z^2 \mapsto \sum_{m' \in \Z^2} \eta_\eps(m')\eta_\eps(m-m')$.\\

We deduce from~\eqref{e.sfeta_et_sfkappa} and~\eqref{e.convol} that $\sfF (\eta_\eps * \eta_\eps  ) = \sfF(\kappa_\eps)$. Since, by the dominated convergence theorem, the Fourier coefficients of $\sfF ( \eta_\eps * \eta_\eps )$ are the $ \eta_\eps * \eta_\eps (m)$'s and the Fourier coefficients of $\sfF(\kappa_\eps)$ are the $\kappa_\eps(m)$'s, we obtain that:
\[
\eta_\eps * \eta_\eps = \kappa_\eps \, .
\]
This is equivalent to saying that $(\eps m_1, \eps m_2) \in \eps \Z^2 \longmapsto \eta_\eps(m_1-m_2)$ is a symmetric square root of $K^\eps$.
\end{proof}

The proof of Lemma~\ref{l.sqrt} is split into two sub-lemmas:

\begin{lemma}\label{l.sqrt.2}
Assume that $f$ satisfies Condition~\ref{a.super-std} and Condition~\ref{a.std} for $\alpha > 5$. In this lemma, we work with the function:
\[
\rho^\eps \, : \, \xi \in \R^2 \mapsto \lambda(\eps\xi) = \sqrt{\sum_{m\in\Z^2}\hat{\kappa}(\xi-2\pi \eps^{-1} m))} \, .
\]
For each $\eps>0$, let:
\[
\Upsilon(\eps)=\max\Big(\int_{\eps^{-1}\T^2}|\rho^\eps(\xi)|d\xi,\int_{\eps^{-1}\T^2}|\Delta \rho^\eps(\xi)|d\xi,\int_{\eps^{-1}\T^2}|\partial_1\Delta \rho^\eps(\xi)|+|\partial_2\Delta \rho^\eps(\xi)|d\xi\Big).\]
Then, for every $\eps > 0$, $\Upsilon(\eps) < +\infty$. Moreover, there exist an absolute constants $C_1<+\infty$ and a constant $\eps_1=\eps_1(\kappa) > 0$ such that for every $\eps \in ]0,\eps_1]$ we have:
  \[
  \sum_{m\in\Z^2}|\eta_\eps(m)|\leq C_1\Upsilon(\eps) \frac{1}{\eps} \log \left(\frac{1}{\eps} \right) \, .
  \]
\end{lemma}

\begin{lemma}\label{l.sqrt.3}
Assume that $\kappa(x)=e^{-\frac{1}{2}|x|^2}$. Then, there exists $\eps_2>0$ such that $\sup_{\eps \in ]0,\eps_2]}\Upsilon(\eps) < +\infty$.
\end{lemma}

Lemma~\ref{l.sqrt.2} follows easily from appropriate integration by parts. The proof of Lemma~\ref{l.sqrt.3} is just straightforward elementary computation and uses very crude estimates throughout.

\begin{proof}[Proof of Lemma \ref{l.sqrt.2}]
First note that $\Upsilon(\eps) < +\infty$ comes from the fact that (by Lemma~\ref{l.sqrt.0}) $\lambda_\eps \in C^3(\T^2)$, hence $\rho^\eps \in C^3(\eps^{-1}\T^2)$. Next, note that by an obvious change of variable, we have:
\[
\eta_\eps(m) =\frac{\eps}{2\pi}\int_{\eps^{-1}\T^2}e^{i\langle\eps m,\xi\rangle}\rho^\eps(\xi) \, d\xi \, .
\]
Hence, $\eta_\eps(0) \leq \frac{\eps}{2\pi}\int_{\eps^{-1}\T^2}|\mu_\eps(\xi)|d\xi \leq \frac{\eps}{2\pi} \Upsilon(\eps)$. Now, let $m \neq 0$. By integration by parts, we have:
\begin{equation}\label{e.sqrt_split_estimates_1}
  \eta_\eps(m)=\frac{1}{2\pi\eps |m|^2}\int_{\eps^{-1}\T^2}\Delta \rho^\eps(\xi)e^{i\langle\eps m,\xi\rangle}d\xi \, ,
\end{equation}
and:
\begin{equation}\label{e.sqrt_split_estimates_2}
  \eta^\eps(m)=\frac{1}{i2\pi\eps^2 |m|^2(m_1\pm m_2)}\int_{\eps^{-1}\T^2}(\partial_1\pm\partial_2)\Delta \rho^\eps(\xi)e^{i\langle\eps m,\xi\rangle}d\xi \, .
\end{equation}
In~\eqref{e.sqrt_split_estimates_2}, since $m\neq 0$, at least one of the two expressions is well defined. Thus, for each $m\in\Z^2$:
\[
|\eta_\eps(m)|\leq\frac{1}{\eps}\frac{1}{2\pi}\Upsilon(\eps)\min\Big(\eps^2,\frac{1}{|m|^2},\frac{1}{\eps|m|^2(|m_1|+|m_2|)}\Big) \, .
\]
This implies that:
\begin{eqnarray*}
  \sum_{m \in \Z^2} |\eta_\eps(m)|&=& |\eta_\eps(0)|+\sum_{m \in \Z^2, \, 0<|m|\leq\eps^{-1}} |\eta_\eps(m)| + \sum_{m \in \Z^2, \, |m|>\eps^{-1}} |\eta_\eps(m)|\\
  &\leq&\frac{1}{\eps}\frac{1}{2\pi} \Upsilon(\eps) \left( \eps+\frac{1}{\eps}\sum_{0<|m|\leq\eps^{-1}}\frac{1}{|m|^2}+\frac{1}{\eps^2}\sum_{|m|>\eps^{-1}}\frac{1}{|m|^2(|m_1|+|m_2|)}\right)\\
  &\leq& C'' \Upsilon(\eps) \left(\eps+\frac{1}{\eps}\log \left( \frac{1}{\eps} \right)+ \frac{1}{\eps} \right)\, ,
\end{eqnarray*}
for some $C''<+\infty$, all this being valid for small enough values of $\eps>0$.
\end{proof}

\begin{proof}[Proof of Lemma \ref{l.sqrt.3}]
Since $\kappa(x)=e^{-\frac{1}{2}|x|^2}$ the Fourier transform of $\kappa$ is
\[
\hat{\kappa} : \xi \in \R^2 \longmapsto \frac{1}{4\pi^2}\int_{\R^2} e^{-i\langle\xi,x\rangle}\kappa(x)dx=\frac{1}{2\pi}e^{-\frac{1}{2}|\xi|^2} \, ,
\]
so that
\[
\rho^\eps(\xi)=\sqrt{\sum_{m\in\Z^2}e^{-\frac{1}{2}|\xi-2\pi\eps^{-1}m|^2}} \, .
\]
Let $P_1=Id$, $P_2=\Delta$ and $P_3=2\partial_1\Delta$. Then:
  \[\Upsilon(\eps)=\max\Big(\int_{\eps^{-1}\T^2}|P_1\rho^\eps(\xi)|d\xi,\int_{\eps^{-1}\T^2}|P_2\rho^\eps(\xi)|d\xi,\int_{\eps^{-1}\T^2}|P_3\rho^\eps(\xi)|d\xi\Big).\]
  For the last argument of the max we use the fact that $\rho^\eps$ remains unchanged when switching the two coordinates of $\R^2$. We begin by justifying the following claim:

\begin{claim}\label{cl.algebra}
There exists $C'<+\infty$ such that for each $\xi\in\eps^{-1}\T^2$ and each $j\in\lbrace 1,2,3\rbrace$,
  \[|P_j\rho^\eps(\xi)|\leq C' \left( \sum_{m\in\Z^2}e^{-\frac{1}{13}|\xi-2\pi\eps^{-1}m|^2} \right)^3 \, .\]
\end{claim}
\begin{proof}
Elementary algebra shows that for each $j\in\lbrace 1,2,3\rbrace$, there exists a polynomial function $Q_j:\R^2\times\R^2\times\R^2\rightarrow\R$ of degree at most $j$ such that for each $\xi\in\eps^{-1}\T^2$, $\rho^\eps(\xi)^5P_j\rho^\eps(\xi)$ equals:
\begin{equation}\label{e.sqrt_derivatives}
\sum_{m_1,m_2,m_3\in\Z^2}Q_j(\xi-2\pi\eps^{-1}m_1,\xi-2\pi\eps^{-1}m_2,\xi-2\pi\eps^{-1}m_3)\prod_{i=1}^3e^{-\frac{1}{2} |\xi-2\pi\eps^{-1}m_i|^2} \, .
\end{equation}
By using the very crude bound $\forall m \in \Z^2, \, \rho^\eps(\xi) \geq e^{-\frac{1}{4}|\xi-2\pi\eps^{-1}m|^2}$, we obtain that for each $\xi\in\eps^{-1}\T^2$, $|P_j\rho^\eps(\xi)|$ is at most:
\begin{align*}
& \sum_{m_1,m_2,m_3 \in \Z} |Q_j(\xi-2\pi\eps^{-1}m_1,\xi-2\pi\eps^{-1}m_2,\xi-2\pi\eps^{-1}m_3)|  \prod_{i=1}^3 e^{- \left( \frac{1}{2} - \frac{5}{3} \times \frac{1}{4} \right) |\xi-2\pi\eps^{-1}m_i|^2}\\
& = \sum_{m_1,m_2,m_3 \in \Z} |Q_j(\xi-2\pi\eps^{-1}m_1,\xi-2\pi\eps^{-1}m_2,\xi-2\pi\eps^{-1}m_3)| \prod_{i=1}^3 e^{-\frac{1}{12}|\xi-2\pi\eps^{-1}m_i|^2} \, .
\end{align*}
Now, note that there exists a constant $C'<+\infty$ such that, for each $j \in \{ 1,2,3\}$, for each $m_1,m_2,m_3\in\Z^2$, and for each $\xi\in\R^2$:
\[
\Big|Q_j(\xi-2\pi\eps^{-1}m_1,\xi-2\pi\eps^{-1}m_2,\xi-2\pi\eps^{-1}m_3)\prod_{i=1}^3e^{-\frac{1}{12}|\xi-2\pi\eps^{-1}m_i|^2}\Big| \leq C'\prod_{i=1}^3e^{-\frac{1}{13}|\xi-2\pi\eps^{-1}m_i|^2} \, ,
\]
and we are done.
\end{proof}
Let us use the claim to conclude. Let $m\in\Z^2$ be such that $|m|\geq 2$ and $\xi\in[-\eps^{-1}\pi,\eps^{-1}\pi]^2$. Then, $|\xi-2\pi\eps^{-1}m|\geq\pi\eps^{-1}|m|$. Therefore:
\begin{eqnarray*}
\sum_{m\in\Z^2,\ |m|\geq 2}e^{-\frac{1}{13}|\xi-2\pi\eps^{-1}m|^2} & \leq & \sum_{m\in\Z^2,\ |m|\geq 2} e^{-\frac{\pi}{13}\eps^{-2}|m|^2}\\
& \leq & C'' e^{-\frac{4\pi}{13}\eps^{-2}} \, ,
\end{eqnarray*}
for some $C'' < +\infty$ and if $\eps$ is sufficiently small. Moreover, $\sum_{m\in\Z^2,\ |m|\leq 1}e^{-\frac{1}{13}|\xi-2\pi\eps^{-1}m|^2} \leq 5$, hence by expanding the cubed sum in Claim~\ref{cl.algebra}, we obtain that if $\eps$ is sufficiently small, then:
\[
|P_j\rho^\eps(\xi)|\leq C''' \left( e^{-\frac{4\pi}{13}\eps^{-2}} + \sum_{m\in\Z^2,\ |m|\leq 1}e^{-\frac{1}{13}|\xi-2\pi\eps^{-1}m|^2} \right) \, ,
\]
for some $C''' < +\infty$. Finally:
\begin{eqnarray*}
\Upsilon(\eps) & \leq & C''' \int_{\eps^{-1} \T^2} \left( e^{-\frac{4\pi}{13}\eps^{-2}} + \sum_{m\in\Z^2,\ |m|\leq 1}e^{-\frac{1}{13}|\xi-2\pi\eps^{-1}m|^2} \right) d\xi\\
& \leq & C''' (2\pi\eps)^2 e^{-\frac{4\pi}{13}\eps^{-2}} +  5 C''' \int_{\R^2} e^{-\frac{1}{13}|\xi|^2} d\xi \, ,
\end{eqnarray*}
which is less than some absolute constant since we consider only $\eps$ small.
\end{proof}

\section{Sprinkling discretization scheme}\label{s.sprinkling}

In this section, we prove Proposition~\ref{p.sprinkling}. We do not rely on arguments from other sections. Recall that $\mathcal{T}=(\mathcal{V},\mathcal{E})$ is the square face-centered lattice and that $\mathcal{T}^\eps=(\mathcal{V}^\eps,\mathcal{E}^\eps)$ denotes $\calT$ scaled by $\eps$. Given an edge $e=(x,y)$, we take the liberty of writing ``$z\in e$'' as a shorthand for ``$\exists t\in[0,1]$ such that $z=ty+(1-t)x$''. For each $R>0$ and $\eps>0$, let $\mathcal{T}^\eps_R=(\calE^\eps_R,\calV^\eps_R)$ denote the sublattice of $\cal{T}^\eps$ generated by the edges $e$ that intersect $[0,2R] \times [0,R]$.\\

In this section, we never use the fact that our Gaussian field is non-degenerate. As we shall see at the end of this section, Proposition~\ref{p.sprinkling} is an easy consequence of the following approximation estimate.

\begin{proposition}\label{p.folds}
Assume that $f$ satisfies Condition \ref{a.std} and that it is a.s. $C^2$. Let $p>0$. Given $\eps>0$ and $e=(x,y) \in \calE^\eps$, we call $Fold(e)$ the event that there exists $z\in e$ such that $f(x)\geq-\frac{p}{2}$, $f(y)\geq -\frac{p}{2}$, and $f(z)<-p$. There exist constants $c_2=c_2(\kappa,p)>0$ and $\eps_0=\eps_0(\kappa,p)>0$ such that for each $\eps \in ]0,\eps_0]$ we have:
\[
\forall e \in \calE^\eps, \, \prob\left[\fold(e)\right]\leq C_1 \exp\left(-c_2\eps^{-4}\right) \, .\] 
\end{proposition}
A key ingredient in proving this inequality will be the Borell-Tsirelson-Ibragimov-Sudakov (or BTIS) inequality (see Theorem 2.9 of \cite{azws}).
\begin{proof}[Proof of Proposition \ref{p.folds}]
Let us fix $e=(x,y)\in\mathcal{E}^\eps$ and consider the vector $v$ defined by $\eps v=y-x$. On the event $\fold(e,\eps)$, by Taylor's inequality applied to $f$ between points $x$, $z$ and $y$, there exist $w_1,w_2\in e$ such that $\partial_v f(w_1)>\frac{p|v|}{2\eps}$ and $\partial_v f(w_2)<-\frac{p|v|}{2\eps}$. Applying Taylor's estimate to $\partial_vf$ between $w_1$ and $w_2$ we conclude that there exists $w_3\in e$ such that $|\partial_{v,v}^2f(w_3)|>\frac{p|v|}{\eps^2}$. Hence:
\[
\Pro \left[ \fold(e)\right]\leq \Pro \left[ \sup_{w\in e}|\partial_{v,v}^2f(w)|>\frac{p|v|}{\eps^2} \right] \, .
\]
Let $x_t=\eps tv$ and $g_\eps^v(t)=\partial^2_{v,v}f(x_t)$. By translation invariance of $f$, we have:
\[
\Pro \left[ \sup_{w\in e}|\partial^2_{v,v}f(w)|>\frac{p|v|}{\eps^2} \right] = \prob\left[ \sup_{t\in[0,1]}|g_\eps^v(t)|>\frac{p|v|}{\eps^2}\right] \, .
\]
The strategy is to apply the BTIS inequality to $g_\eps^v$. Note that $g_\eps^v$ is a centered Gaussian field on $[0,1]$ which is a.s. bounded. Hence, by Theorem~2.9 of~\cite{azws}, $\E \left[ | \sup_{t\in[0,1]} g_\eps^v(t) | \right]<+\infty$. Note that $\E \left[ | \sup_{t\in[0,1]} g_\eps^v(t) | \right]$ is non-decreasing in $\eps$, let $C_4=\max_{v' \in \Gamma}\E \left[ | \sup_{t\in[0,1]} g_1^{v'}(t)| \right]$, and choose $\eps_0 \in ]0,1]$ sufficiently small so that $\min_{v'\in\Gamma}\frac{p|v'|}{\eps_0^2}>2C_4$. Note that, by translation invariance:
      \[
      \sigma^2:=\sup_{t\in[0,1]}\var (g_\eps^v(t))=\var(g^v_\eps(0))=\partial^4_{v,v,v,v}\kappa(0) \, .
      \]
If $\partial^4_{v,v,v,v}\kappa(0)=0$, then $g_\eps^v\equiv0$ a.s. and we are done. Assume now that $\partial^4_{v,v,v,v}\kappa(0) > 0$. Then, by the BTIS inequality (see Theorem 2.9 in \cite{azws}), for each $\eps\in]0,\eps_0]$:
        \begin{eqnarray*}
        \prob\left[ \sup_{t\in[0,1]}|g_\eps^v(t)|>\frac{p|v|}{\eps^2}\right] & \leq & 2 \, \prob \left[ \sup_{t\in[0,1]} g_\eps^v(t) >\frac{p|v|}{\eps^2} \right]\\
        & \leq & 2 \, \prob \left[ \Big| \sup_{t\in[0,1]} g_\eps^v(t) - \E [ \sup_{t\in[0,1]} g_\eps^v(t) ] \Big|  > \frac{p|v|}{2\eps^2} \right]\\ & \leq & 4 \, \exp\left(-\frac{1}{2\sigma^2}\left(\frac{p|v|}{2\eps^2}\right)^2\right) \, .
        \end{eqnarray*}
        Since $v$ can take only finitely many values, we have obtain what we want.
\end{proof}

\begin{proof}[Proof of Proposition \ref{p.sprinkling}]
By Proposition~\ref{p.folds} (and by a simple union bound), it is enough to prove that $\cross^\eps_{p/2}(2R,R) \cap  (\cup_e \fold(e))^c  \subseteq \cross_p(2R,R)$, where the union is over each $e \in \calE^\eps_R$. Assume that for every $e \in \calE$ $\fold(e)$ is not satisfied. Then, for each edge $e=(x,y)\in\calE^\eps_R$ which is colored black in the discrete model of parameter $p/2$, and for each $z\in e$, we have $f(z)\geq -p$. In other words, each black edge is contained in $\calD_p$. If in addition $\cross^\eps_{p/2}(2R,R)$ is satisfied, then there exists a crossing of $[0,2R] \times [0,R]$ from left to right made up of black edges. This crossing belongs to $\calD_p$ so that $\cross_p(2R,R)$ is satisfied.
\end{proof}

\bibliographystyle{alpha}
\bibliography{ref_perco}

\ \\
{\bf Alejandro Rivera}\\
Univ. Grenoble Alpes\\
UMP5582, Institut Fourier, 38000 Grenoble, France\\
alejandro.rivera@univ-grenoble-alpes.fr\\
Supported by the ERC grant Liko No 676999\\

\medskip
\ni
{\bf Hugo Vanneuville} \\
Univ. Lyon 1\\
UMR5208, Institut Camille Jordan, 69100 Villeurbanne, France\\
vanneuville@math.univ-lyon1.fr\\
\url{http://math.univ-lyon1.fr/~vanneuville/}\\
Supported by the ERC grant Liko No 676999\\

\end{document}